\documentclass[letterpaper,10pt]{scrartcl}

\usepackage[utf8]{inputenc}
\usepackage{amssymb}
\usepackage{amsfonts}
\usepackage{graphicx}
\usepackage{multirow}
\usepackage{bigstrut}
\usepackage{epstopdf}
\usepackage[english]{babel}
\usepackage{combelow}
\usepackage{amsmath}
\usepackage{amsthm}
\usepackage{algorithm}
\usepackage{algorithmicx}
\usepackage[noend]{algpseudocode}
\usepackage{xcolor}
\usepackage{bm}
\usepackage{tikz}
\usetikzlibrary{shapes}
\usetikzlibrary{arrows}
\usetikzlibrary{positioning}
\usetikzlibrary{matrix}
\usetikzlibrary{patterns}
\usetikzlibrary{decorations.pathreplacing,decorations.pathmorphing,decorations.markings}
\usepackage{tikz-qtree}
\usetikzlibrary{trees,calc,arrows.meta,positioning,bending}
\usepackage{mathtools}
\usepackage{todonotes}
\usepackage{sidecap}
\usepackage{fullpage}
\usepackage{hyperref}

\newtheorem{proposition}{Proposition}
\newtheorem{lemma}{Lemma}
\newtheorem{assumption}{Assumption}

\newtheorem{rmk}{Remark}
\numberwithin{equation}{section}

\newcommand{\E}{\mathbb{E}}

\newcommand{\Xcal}{\mathcal{X}}
\newcommand{\Ycal}{\mathcal{Y}}

\newcommand{\bftheta}{\boldsymbol{\theta}}
\newcommand{\bfTheta}{\boldsymbol{\Theta}}

\definecolor{amaranth}{rgb}{0.9, 0.17, 0.31}
\definecolor{cobalt}{rgb}{0.0, 0.28, 0.67}
\definecolor{amber}{rgb}{1.0, 0.49, 0.0}

\newcommand{\kibitz}[2]{\ifnum\Comments=1\textcolor{#1}{#2}\fi}

\graphicspath{{figures/}}
\DeclareGraphicsExtensions{.pdf,.eps,.png,.jpg,.jpeg}

\ifpdf
\hypersetup{
  pdftitle={Title}, 
  pdfauthor={Name} 
}
\fi

\title{Context-aware learning of hierarchies of low-fidelity models for multi-fidelity uncertainty quantification}

\author{Ionu\cb{t}-Gabriel Farca\cb{s}\thanks{Oden Institute for Computational Engineering and Sciences, The University of Texas at Austin}, Benjamin Peherstorfer\thanks{Courant Institute of Mathematical Sciences, New York University}, Frank Jenko\thanks{Max Planck Institute for Plasma Physics}, \\ Tobias Neckel\thanks{Department of Informatics, Technical University of Munich}, and Hans-Joachim Bungartz\footnotemark[4]}

\begin{document}

\maketitle

\begin{abstract}

Multi-fidelity Monte Carlo methods leverage low-fidelity and surrogate models for variance reduction to make tractable uncertainty quantification even when numerically simulating the physical systems of interest with high-fidelity models is computationally expensive. 
This work proposes a context-aware multi-fidelity Monte Carlo method that optimally balances the costs of training low-fidelity models with the costs of Monte Carlo sampling. It generalizes the previously developed context-aware bi-fidelity Monte Carlo method to hierarchies of multiple models and to more general types of low-fidelity models.
When training low-fidelity models, the proposed approach takes into account 
the context in which the learned low-fidelity models will be used, namely for variance reduction in Monte Carlo estimation, which allows it to find optimal trade-offs between training and sampling to minimize upper bounds of the mean-squared errors of the estimators for given computational budgets. This is in stark contrast to traditional surrogate modeling and model reduction techniques that construct low-fidelity models with the primary goal of approximating well the high-fidelity model outputs and typically ignore the context in which the learned models will be used in upstream tasks. The proposed context-aware multi-fidelity Monte Carlo method applies to hierarchies of a wide range of types of low-fidelity models such as sparse-grid and deep-network models. Numerical experiments with the gyrokinetic simulation code \textsc{Gene} show speedups of up to two orders of magnitude compared to standard estimators when quantifying uncertainties in small-scale fluctuations in confined plasma in fusion reactors. 
This corresponds to a runtime reduction from 72 days to about four hours on one node of the Lonestar6 supercomputer at the Texas Advanced Computing Center.

\end{abstract}  

\section{Introduction} \label{sec:Intro}

Uncertainty quantification is an essential building block for achieving predictive numerical simulations of physical systems. 
To make accurate Monte Carlo estimation of uncertainties tractable even when simulations of the physical system of interest are computationally expensive, multi-fidelity methods rely on low-fidelity or surrogate models: the low-fidelity models are leveraged for variance reduction to achieve speedups and the high-fidelity models are occasionally evaluated to guarantee unbiasedness; see \cite{PWG18} for a survey on multi-fidelity methods. 
However, if low-fidelity models are not readily available, then they need to be constructed and trained first, which can incur additional computational costs and require additional evaluations of the high-fidelity models to generate training data.

In this work, we build on the context-aware bi-fidelity Monte Carlo method introduced in \cite{Pe19} and propose the context-aware multi-fidelity Monte Carlo (CA-MFMC) method that trades off the costs of training hierarchies of multiple low-fidelity models with the costs of Monte Carlo sampling to obtain multi-fidelity estimators that minimize upper bounds of the mean-squared errors for given computational budgets. 
The proposed approach is context-aware \cite{AP20,Farcas2020,Pe19,WP22,9867770} in the sense that the low-fidelity models are trained to maximize variance reduction in Monte Carlo estimation. This means that the context in which the learned models will be used, namely for variance reduction in Monte Carlo estimation, is taken into account during training, which distinguishes it from traditional surrogate modeling and model reduction techniques that construct low-fidelity models with the primary goal of approximating well the high-fidelity model outputs and typically ignore the context in which the learned models are ultimately used \cite{SIREV,RozzaPateraSurvey}.
Our proposed CA-MFMC method can be combined with a wide range of types of data-fit low-fidelity models models such as sparse-grid-based models and deep-network models. 
We show that CA-MFMC achieves speedups of up to two orders of magnitude compared to single-fidelity Monte Carlo and standard multi-fidelity estimators when quantifying uncertainties in a plasma micro-turbulence scenario from the ASDEX Upgrade Experiment\footnote{\url{https://www.ipp.mpg.de/16195/asdex}}.

We build on multi-fidelity Monte Carlo (MFMC) estimators \cite{KGW21,NME:NME4761, PWK16MFMCAsymptotics,PWG16,PWG18,Qi18,Gr22,egusphere-2022-797} that leverage a given hierarchy of low-fidelity models for variance reduction. 
Several works have employed MFMC estimators to quantify uncertainties in plasma micro-turbulence simulations \cite{Farcas2020,Ko22} and for estimating statistics in collisionless energetic particle confinement in stellerators \cite{LCP22}. 
See also \cite{FMJ22} for other techniques for uncertainty quantification in plasma simulations. None of these methods trade off the training of low-fidelity models with sampling, however.
The works \cite{Go20,doi:10.1137/19M1263534,doi:10.1137/20M1321607} formulate generalized control variate techniques for multi-fidelity uncertainty propagation. There is a wide range of other multi-fidelity techniques that are based on other concepts than control variates and aim to find surrogate models from multi-fidelity information sources such as the collocation approach introduced in \cite{doi:10.1137/130929461,HAMPTON2018315} and the multi-fidelity networks proposed in \cite{Gorodetsky_2020,Gorodetsky2021}; see also \cite{Newberry2022,De_2020,RAZI2019992,doi:10.1137/21M1408312}. 
There have been extensions to the multilevel Monte Carlo method \cite{cliffe_multilevel_2011,Gi08,teckentrup_further_2013} that consider spatially-adaptive mesh refinements and optimal mesh hierarchies \cite{Collier2015,doi:10.1137/15M1016448,Haji-Ali2016}.
However, in that line of work, the low-fidelity models are coarse-grid approximations and thus the costs of constructing low-fidelity models are typically considered to be negligible and therefore are ignored. In contrast, we consider more general types of low-fidelity models such as data-fit models that incur training costs.

The first work that considered trading off training low-fidelity models and multi-fidelity Monte Carlo estimation is \cite{Pe19}, which studies the bi-fidelity setting in which the low-fidelity model has algebraic accuracy and cost rates. In \cite{https://doi.org/10.48550/arxiv.2201.10745} a similar trade off is considered in the bi-fidelity case but with more general cost and error rates and a specific focus on polynomial-chaos-based surrogate models. We go far beyond by introducing the CA-MFMC estimator based on context-aware learning that learns hierarchies of more than one low-fidelity model for variance reduction. 
We first extend the work on a single low-fidelity model with algebraic cost and error rates in \cite{Pe19} to apply to low-fidelity models with more general rates. Key is that the correspond bounds can be nested, which motivates the sequential training of low-fidelity models to obtain a hierarchy. 
In the proposed sequential training approach to fit hierarchies of low-fidelity models for CA-MFMC estimators, each step leads to an optimal trade off between training and sampling. 
This leads to a context-aware learning approach because the models are learned such that the CA-MFMC estimator achieves a low mean-squared error, rather than being trained to accurately approximate high-fidelity model outputs as in traditional model reduction.

To demonstrate the performance of the proposed CA-MFMC estimator, we apply it to quantifying uncertainties in plasma micro-turbulence simulations motivated by the ITER experiment\footnote{\url{https://www.iter.org}}. 
The goal of the ITER experiment is to create, for the very first time, a self-sustained plasma in the laboratory. 
A physics obstacle are the above-mentioned small-scale fluctuations which cause energy loss rates despite sophisticated plasma confinements via strong and shaped magnetic fields. 
Building on the gyrokinetic simulation code GENE \cite{Je00}, we show that the proposed estimator achieves speedups of up to two orders of magnitude compared to single-fidelity Monte Carlo and standard multi-fidelity estimators, which translates into a runtime reduction from 72 days to about four hours on one node of the Lonestar6 supercomputer at the Texas Advanced Computing Center\footnote{\url{https://www.tacc.utexas.edu/systems/lonestar6}}.

The remainder of this paper is organized as follows.
Section~\ref{sec:background} introduces the notation and summarizes the traditional MFMC algorithm \cite{NME:NME4761, PWG16} and the bi-fidelity context-aware algorithm formulated in \cite{Pe19}.
Section~\ref{sec:camfmc} introduces our context-aware learning approach for low-fidelity models with general cost/error rates and multiple low-fidelity models for multi-fidelity sampling methods.
Section~\ref{sec:results} presents numerical results in two scenarios: a heat conduction problem defined on a two-dimensional spatial domain with nine uncertain parameters and a realistic plasma micro-turbulence scenario with $12$ uncertain inputs, for which one realization of the uncertain inputs requires a total runtime of about $411$ seconds on $32$ cores.
The code and data to reproduce our numerical results are available at \texttt{https://github.com/ionutfarcas/context-aware-mfmc}.
\section{Preliminaries} \label{sec:background}
This section reviews traditional MFMC estimators \cite{NME:NME4761,PWG16} and the bi-fidelity context-aware MFMC estimator \cite{Pe19} that uses a high-fidelity model and a single low-fidelity model only.

\subsection{Static multi-fidelity Monte Carlo estimation} \label{subsec:static_MFMC}
Let $f^{(0)}: \Xcal \to \Ycal$ represent the input-output response of a potentially expensive-to-evaluate computational model. The input domain is $\Xcal \subset \mathbb{R}^d$ and the output is scalar with output domain $\Ycal \subset \mathbb{R}$.
We consider the situation where the input $\boldsymbol{\theta} = [\theta_1, \theta_2, \ldots, \theta_d]^T$ is a realization of a random variable $\bfTheta = [\Theta_1, \dots, \Theta_d]^T$ so that the output $f^{(0)}(\bftheta)$ becomes a realization of a random variable $f^{(0)}(\bfTheta)$ too. 
We denote the probability density function of $\bfTheta$ as $\pi$. 
Our goal is to estimate the expected value of $f^{(0)}(\bfTheta)$
    \begin{equation*} 
        \mu_0 = \E[f^{(0)}(\bfTheta)] = \int_{\Xcal} f^{(0)}(\boldsymbol{\theta}) \pi(\boldsymbol{\theta}) \mathrm{d}\boldsymbol{\theta}\,.
    \end{equation*}

The MFMC estimator introduced in \cite{NME:NME4761,PWG16} combines a hierarchy of $k + 1$ models, consisting of a high-fidelity model $f^{(0)}$ and $k$ low-fidelity models $f^{(1)}$, $f^{(2)}$, \ldots, $f^{(k)}$. 
The accuracy of the low-fidelity models is measured by their Pearson correlation coefficient with respect to~$f^{(0)}$,
\begin{equation*}
\rho_{j} = \frac{\mathrm{Cov}[f^{(0)}, f^{(j)}]}{\sigma_0\sigma_j} \in [-1, 1], \quad j = 1, \ldots, k,
\end{equation*}
where $\sigma_j^2$ denotes the variance of the output random variable $f^{(j)}(\bfTheta)$ for $j=0, 1, \ldots, k$ and $\operatorname{Cov}[\cdot,\cdot]$ is the covariance.
The evaluation costs of the models are $0 < w_1, \dots, w_k$, respectively.
We normalize the cost of evaluating the high-fidelity model such that $w_0 = 1$, without loss of generality.
The models are assumed to satisfy the following ordering
\begin{equation} \label{eq:MFMC_ordering}
    1 = \left|\rho_0\right| > \left|\rho_1\right| > \ldots > \left|\rho_k\right|, \quad \frac{w_{j-1}}{w_j} > \frac{\rho_{j-1}^2 - \rho_{j}^2}{\rho_{j}^2 - \rho_{j+1}^2}, \quad j = 1, \ldots, k, 
\end{equation}
where $\rho_{\ell} = 0$ for $\ell \geq k + 1$. 

Let now $m_j \in \mathbb{N}$ for $j = 0, 1\ldots, k$ denote the number of evaluations of model $f^{(j)}$, which satisfy
\begin{equation*}
0 \leq m_0 \leq m_1 \leq \ldots \leq m_k
\end{equation*}
because of \eqref{eq:MFMC_ordering}. 
Consider $m_k$ independent and identically distributed (i.i.d.)~samples drawn from the input density $\pi$:
\begin{equation*}
    \boldsymbol{\theta}_1, \boldsymbol{\theta}_2, \ldots, \boldsymbol{\theta}_{m_k}.
\end{equation*}
To obtain the MFMC estimator, model $f^{(j)}$ is evaluated at the first $m_j$ samples $\boldsymbol{\theta}_1, \boldsymbol{\theta}_2, \ldots, \boldsymbol{\theta}_{m_j}$ to obtain output random variables
\begin{equation}\label{eq:Prelim:ModelOutputs}
f^{(j)}(\boldsymbol{\theta}_1), f^{(j)}(\boldsymbol{\theta}_2), \ldots, f^{(j)}(\boldsymbol{\theta}_{m_j})\,,
\end{equation}
for $j = 0, 1, \ldots, k$. 
From \eqref{eq:Prelim:ModelOutputs}, derive the following standard MC estimators for $j = 1, \dots, k$:
\begin{equation*}
\hat{E}_{m_j}^{(j)} = \frac{1}{m_j} \sum_{i=1}^{m_j} f^{(j)}(\boldsymbol{\theta}_i), \quad
\hat{E}_{m_{j - 1}}^{(j)} = \frac{1}{m_{j - 1}} \sum_{i=1}^{m_{j - 1}} f^{(j)}(\boldsymbol{\theta}_i) 
\end{equation*}
as well as
\[
\hat{E}_{m_0}^{(0)} = \frac{1}{m_0}f^{(0)}(\boldsymbol{\theta}_i)\,.
\]
The MFMC estimator is
\begin{equation*}
\hat{E}^{\mathrm{MFMC}} = \hat{E}_{m_0}^{(0)} + \sum_{j=1}^k \alpha_j (\hat{E}_{m_j}^{(j)} - \hat{E}_{m_{j-1}}^{(j)})
\end{equation*}
with the coefficients $\alpha_1, \dots, \alpha_k \in \mathbb{R}$. 
The total computational cost of the MFMC estimator is therefore
\[
p = \sum_{j = 0}^k m_jw_j\,.
\]
Fixing $p$ and then selecting $m_0^*, m_1^*, \ldots, m_k^*$ model evaluations and $\alpha_1^*, \alpha_2^*, \ldots, \alpha_k^*$ coefficients such that the variance of the MFMC estimator $\hat{E}^{\mathrm{MFMC}}$ is minimized, gives the MSE
\begin{equation}\label{eq:mse_mfmc}
    \mathrm{MSE}[\hat{E}^{\mathrm{MFMC}}] = \frac{\sigma_0^2}{p} \left( \sum_{j=0}^k \sqrt{w_j (\rho_{j}^2 - \rho_{j+1}^2)} \right) ^2.
\end{equation}
The optimal $m_0^*, m_1^*, \ldots, m_k^*$ model evaluations and $\alpha_1^*, \alpha_2^*, \ldots, \alpha_k^*$ coefficients are available in closed form for a given budget $p$; see \cite{PWG16} for details.
    
The MFMC estimator leverages the cheaper-to-evaluate low-fidelity models with the aim to achieve a lower MSE than a standard MC estimator with the same costs.
That is, given $m$ i.i.d.~samples from $\boldsymbol{\Theta}$, the standard MC mean estimator is
\begin{equation}\label{eq:mean_std_mc}    \hat{E}_m^{(0)} = \frac{1}{m} \sum_{i=1}^m f^{(0)}(\boldsymbol{\theta}_i).
\end{equation}
The computational cost of the Monte Carlo estimator \eqref{eq:mean_std_mc} of $\mathbb{E}[f^{(0)}(\boldsymbol{\Theta})]$ is $p = mw_0 = m$ because the function $f^{(0)}$ is evaluated at $m$ realizations and the cost of each evaluation of $f^{(0)}$ is $w_0 = 1$. 
The MSE of the standard MC estimator $\hat{E}_n^{(0)}$ reads
\begin{equation}\label{eq:mse_mc}
    \mathrm{MSE}[\hat{E}_m^{(0)}] = \frac{\sigma_0^2}{p} w_0 = \frac{\sigma_0^2}{p}.
\end{equation}

\subsection{Context-aware bi-fidelity Monte Carlo estimator with algebraic accuracy and cost rates}\label{sec:Prelim:AMFMC}
The work \cite{Pe19} introduces a context-aware bi-fidelity MFMC approach that trades off the training costs of constructing a low-fidelity model $f^{(1)}_{n}$ and sampling.  
Following \cite{Pe19}, the low-fidelity model $f^{(1)}_{n}$ is obtained via a training process such as data-fit models and reduced models. 
Correspondingly, the subscript $n$ refers to the number of high-fidelity evaluations that are needed to train $f^{(1)}_{n}$. 
For example, in case of obtaining $f^{(1)}_{n}$ via training a neural network, the subscript $n$ refers to the number of training input-output pairs obtained with the high-fidelity model $f^{(0)}$. 
This also means that the correlation coefficient $\rho_1(n)$ between $f^{(0)}(\bfTheta)$ and $f^{(1)}_{n}(\bfTheta)$ as well as the evaluation cost $w_1(n)$ of $f^{(1)}_{n}$ depend on $n$.

The dependency of the correlation and costs on $n$ are described as follows: to trade off training low-fidelity model and sampling, the work \cite{Pe19} makes the assumption that the correlation coefficient between the high-fidelity output random variable $f^{(0)}(\bfTheta)$ and low-fidelity random variable $f^{(1)}_{n}(\bfTheta)$ is bounded by
\begin{equation*}
    1 - \rho_1^2(n) \leq c_1 n^{-\alpha}
\end{equation*}
with respect to $n$,  where $c_1, \alpha > 0$ are constants. 
The cost $w_1(n)$ of evaluating the low-fidelity model is bounded by
\begin{equation*}
    w_1(n) \leq c_2 n^{\beta}
\end{equation*}
with constants $c_2, \beta > 0$.
    
A budget $p$ corresponds to $p$ high-fidelity model evaluations because we have $w_0 = 1$. Thus, if $n$ high-fidelity model samples are used for training $f^{(1)}_{n}$, then a budget of $p - n$ is left for sampling. 
The corresponding context-aware bi-fidelity estimator is
\[
\hat{E}^{\text{CA-MFMC}}_{n} = \hat{E}_{m_0^*}^{(0)} + \alpha_1^* (\hat{E}_{m_1^*, n}^{(1)} - \hat{E}_{m_0^*, n}^{(1)})
\]
where $m_0^*, m_1^*, \alpha_1^*$ are chosen to minimize the MSE of $\hat{E}^{\text{CA-MFMC}}_{n}$ for a given $p$ and $n$. Consequently, the MSE of $\hat{E}^{\text{CA-MFMC}}_{n}$ depends on $p$ and $n$ 
\[
\operatorname{MSE}(\hat{E}^{\text{CA-MFMC}}_n) = \frac{\sigma_0^2}{p - n} \left( \sqrt{1 - \rho_{1}^2(n)} + \sqrt{w_1(n)\rho_{1}^2(n)} \right) ^2 \,,
\]
which is in contrast to the MFMC estimator $\hat{E}^{\mathrm{MFMC}}$ for which the MSE \eqref{eq:mse_mfmc} depends on $p$ only and is independent of potential training costs of constructing low-fidelity models.

If the budget $p$ is fixed, $n$ can be chosen by minimizing the upper bound of the MSE
\begin{equation}\label{eq:AMFMCBound}
    \operatorname{MSE}(\hat{E}^{\text{CA-MFMC}}_n) \leq \frac{2\sigma_0^2}{p - n}\left( c_1 n^{-\alpha} +  c_2 n^{\beta} \right)\,.
\end{equation}
The work \cite{Pe19} shows that there exists a unique $n^*$ that minimizes \eqref{eq:AMFMCBound} with respect to $n$ for a given $p$ under certain assumptions; however, no closed form expression of $n^*$ is available and thus $n^*$ needs to be found numerically. The optimum $n^*$ is bounded independently of the budget $p$; see \cite{Pe19} for details.
\section{Context-aware multi-fidelity Monte Carlo estimation} \label{sec:camfmc}
This section presents the methodological novelty of this paper.
We introduce context-aware learning for MFMC for low-fidelity models with general cost/error rates and hierarchies of multiple---more than one---low-fidelity models. 

\subsection{Setup with multiple models} \label{subsec:camfmc_prelim}
We now consider multiple low-fidelity models $f_{n_1}^{(1)}, \dots, f_{n_k}^{(k)}$ that take $n_1, \dots, n_k$ evaluations, respectively, of the high-fidelity model to be trained. 
We refer to $n_1, \dots, n_k$ as the number of training samples to fit the low-fidelity models. 
We then define the CA-MFMC estimator as
\begin{equation*}
\hat{E}_{n_1, \dots, n_k}^{\text{CA-MFMC}} = \hat{E}_{m_0^*}^{(0)} + \sum_{j = 1}^k\alpha_j^*\left(\hat{E}_{m_j^*}^{(j)} - \hat{E}_{m_{j - 1}^*}^{(j)}\right)\,,
\end{equation*}
where the $m_0^*, \dots, m_k^*$ and $\alpha_1^*, \dots, \alpha_k^*$ minimize the MSE of $\hat{E}_{n_1, \dots, n_k}^{\text{CA-MFMC}}$ for given $n_1, \dots, n_k$ and budget $p$.  
The MSE of the CA-MFMC estimator is
\begin{equation}\label{eq:mse_camfmc}
\mathrm{MSE}[\hat{E}^{\mathrm{CA-MFMC}}_{n_1, \dots, n_k}] = \frac{\sigma_0^2}{p - \sum_{i = 1}^k n_i} \left( \sum_{j=0}^k \sqrt{w_j(n_j) (\rho_{j}^2(n_j) - \rho_{j+1}^2(n_{j  + 1}))} \right) ^2\,,
\end{equation}
where now the MSE depends on the $n_1, \dots, n_k$ training samples used for constructing the low-fidelity models. 
We are interested in choosing $n_1, \dots, n_k$ to minimize an upper bound of the MSE \eqref{eq:mse_camfmc} so that the training costs given by $n_1, \dots, n_k$ and the costs of taking samples from the high- and the low-fidelity output random variables are balanced.

We make the following assumptions on the accuracy and cost behavior of the low-fidelity models with respect to the number of training samples $n_1, \dots, n_k$: 
\begin{assumption} \label{assump:acc_rates}
    For all $j = 1, \dots, k$, there exists a constant $c_{a, j} > 0$ and a positive, decreasing, 
    at least twice continuously differentiable function $r_{a, j}(n_j) : (0, \infty) \rightarrow (0, \infty)$ such that
    \begin{equation*} 
    1 - \rho_{j}^2(n_j) \leq c_{a, j} r_{a, j}(n_j), \quad j = 1, \ldots, k\,.
    \end{equation*}
\end{assumption}
\begin{assumption} \label{assump:cost_rates}
    For all $j = 1, \dots, k$, there exists a constant $c_{c, j} > 0$ and a positive, increasing, at least twice continuously differentiable function $r_{c, j}(n_j)  : (0, \infty) \rightarrow (0, \infty)$ such that the evaluation costs are bounded as
    \begin{equation*} 
    w_j(n_j) \leq c_{c, j} r_{c, j}(n_j), \quad j = 1, \ldots, k.
    \end{equation*}
\end{assumption}

Assumptions \ref{assump:acc_rates}--\ref{assump:cost_rates} hold for a wide range of typical error and cost rates: the functions $r_{a,j}(n) = n^{-\alpha}, \alpha >0$ and $r_{a,j}(n) = \mathrm e^{-\alpha n}, \alpha > 0$ satisfy Assumption~\ref{assump:acc_rates} for $n \in (0, \infty)$. 
The functions $r_{c,j}(n) = n^{\alpha}, \alpha > 1$ and $r_{c,j}(n) = \mathrm e^{\alpha n}, \alpha > 0$ satisfy Assumption~\ref{assump:cost_rates} for $n \in (0, \infty)$. Additionally, the error and cost rates given here are strictly convex.

\subsection{Context-aware multi-fidelity Monte Carlo sampling with one low-fidelity model} \label{subsubsec:camfmc_one_model_setup}
Let us first consider only one low-fidelity model, which means that $k = 1$. 
The following is novel compared to what was introduced in \cite{Pe19} because here we allow more general error and cost rates than \cite{Pe19}; cf.~Section~\ref{sec:Prelim:AMFMC} where \cite{Pe19} is reviewed.

We obtain the following upper bound of the MSE of the CA-MFMC estimator \eqref{eq:mse_camfmc},
\begin{align}
   \mathrm{MSE}[\hat{E}^{\mathrm{CA-MFMC}}_{n_1}] = & \frac{\sigma_0^2}{p - n_1} \left( \sqrt{1 - \rho_{1}^2(n_1)} + \sqrt{w_1(n_1)\rho_{1}^2(n_1)} \right) ^2 \nonumber \\
    = & \frac{\sigma_0^2}{p - n_1} \left( 1 - \rho_{1}^2(n_1) + w_1(n_1)\rho_{1}^2(n_1) + 2 \sqrt{1 - \rho_1^2(n_1)} \sqrt{w_1(n_1)\rho_{1}^2(n_1)} \right) \nonumber   \\
   \leq & \label{eq:CAMFMC_1D_MSE_upper_bound_mean_inequality} \frac{\sigma_0^2}{p - n_1} \left( 2(1 - \rho_{1}^2(n_1)) + 2w_1(n_1)\rho_{1}^2(n_1) \right)   \\
    \leq& \label{eq:CAMFMC_1D_MSE_upper_bound_rho} \frac{2\sigma_0^2}{p - n_1} \left( 1 - \rho_{1}^2(n_1) + w_1(n_1) \right)   \\
    \leq& \label{eq:CAMFMC_1D_MSE_upper_bound_rates} \frac{2\sigma_0^2}{p - n_1}\left( c_{a,1}r_{a,1}\left( n_1 \right) +  c_{c,1} r_{c, 1}\left( n_1 \right) \right) ,
    \end{align}
where inequality \eqref{eq:CAMFMC_1D_MSE_upper_bound_mean_inequality} is due to the means inequality $\forall a, b \in \mathbb{R}_+ \Rightarrow 2\sqrt{a}\sqrt{b} \leq a + b$,  inequality \eqref{eq:CAMFMC_1D_MSE_upper_bound_rho} results from $\rho_1^2 \leq 1$ and \eqref{eq:CAMFMC_1D_MSE_upper_bound_rates} is due to Assumptions \ref{assump:acc_rates} and \ref{assump:cost_rates}.

The objective that we want to minimize with respect to $n_1$ is
\begin{equation}\label{eq:CAMFMC_obj_1D}
    u_1 : [1, p-1] \rightarrow (0, \infty), n_1 \mapsto \frac{1}{p - n_1}\left( c_{a,1}r_{a,1}\left( n_1 \right) +  c_{c,1} r_{c, 1}\left( n_1 \right) \right)\,.
\end{equation}
The following proposition clarifies when a unique global minimum exists.

\begin{proposition}\label{prop:obj_1D_convexity}
Consider function $u_1$ defined in \eqref{eq:CAMFMC_obj_1D} with $r_{a,1}$ and $r_{c,1}$ satisfying Assumptions~\ref{assump:acc_rates} and~\ref{assump:cost_rates}, respectively. Additionally assume that
\begin{equation}
c_{a,1}r_{a,1}^{\prime\prime}(n_1) + c_{c,1}r_{c,1}^{\prime\prime}(n_1) > 0
\label{eq:ConvexityAuxiliaryAsm}
\end{equation}
holds for all $n_1 \in [1, p - 1]$, where $r_{a,1}^{\prime\prime}$ and $r_{c,1}^{\prime\prime}$ are the second derivatives of $r_{a, 1}$ and $r_{c,1}$, respectively. Then, the function $u_1$ has a unique global minimum $n_1^* \in [1, p - 1]$.
\end{proposition}
\begin{proof}
Define the function $h_1(n_1) = c_{a,1}r_{a,1}(n_1) + c_{c,1}r_{c,1}(n_1)$ and notice that 
\begin{equation}
h_1^{\prime\prime}(n_1) = c_{a,1}r_{a,1}^{\prime\prime}(n_1) + c_{c,1}r_{c,1}^{\prime\prime}(n_1) > 0\,,\qquad n_1 \in [1, p - 1]
\label{eq:StrictConvexH}
\end{equation}
due to \eqref{eq:ConvexityAuxiliaryAsm}, which means that $h_1$ is strictly convex in $[1, p - 1]$. 
We now consider $u_1$ and its first derivative
\begin{equation}
u_1^{\prime}(n_1) = \frac{(p - n_1)h^{\prime}_1(n_1) + h_1(n_1)}{(p - n_1)^2}\,.
\label{eq:proof:u1Prime}
\end{equation}

Case 1: If $u_1'(n_1) > 0$ for all $n_1 \in [1, p - 1]$, then $u_1$ is strictly increasing and the minimum of $u_1$ in $[1, p - 1]$ is taken at the left boundary $n_1^* = 1$, which is the global unique minimum.

Case 2: Analogously to Case 1, if $u_1'(n_1) < 0$ for all $n_1 \in [1, p - 1]$, then $u_1$ is strictly decreasing and the unique minimum is $p - 1$. 

Case 3: There is a stationary point $n_1^* \in [1, p - 1]$ such that $u_1'(n_1^*) = 0$. At a stationary point $n_1^*$ of $u_1$, the second derivative,
\[
u_1''(n_1) = \frac{h_1''(n_1)}{p - n_1} + \frac{2}{p - n_1}u_1'(n_1)\,,
\]
is positive $u_1''(n_1^*) > 0$ because of \eqref{eq:StrictConvexH} and $u'_1(n_1^*) = 0$. Thus, together with the fact that function $u_1$ is twice continuously differentiable and univariate, we have that $u''_1$ is positive in a neighborhood about all stationary points, because the inequality $h_1''(n_1^*) > 0$ is strict. Thus, in the interior $(1, p - 1)$, there can only be minima, which also means that a the boundaries of the interval $[1, p - 1]$ can only be maxima. This means that there exists only one minimum, which shows uniqueness. 
\end{proof}

\begin{rmk}
Assumption~\eqref{eq:ConvexityAuxiliaryAsm} in Proposition~\ref{prop:obj_1D_convexity} is satisfied if $r_{a,1}$ is strictly convex and $r_{c,1}$ is convex, which holds for the examples given in Section~\ref{subsec:camfmc_prelim}.
\end{rmk}

We will next show that that the minimizer of \eqref{eq:CAMFMC_obj_1D} is also bounded from above, independent of the computational budget, $p$, which implies that after a finite number of samples, all high-fidelity model samples should be used in the Monte Carlo estimator rather than being used to improve the low-fidelity model.

\begin{proposition} \label{prop:obj_1D_n_star_boundness}
Let Proposition~\ref{prop:obj_1D_convexity} apply with condition \eqref{eq:ConvexityAuxiliaryAsm} satisfied for all $n_1 \in (0, \infty)$. Additionally, assume that there exists an $\bar{n}_1 \in (0, \infty)$ such that
\begin{equation}\label{eq:Uniqueness:ExistStat}
c_{a,1}r_{a,1}'(\bar{n}_1) + c_{c,1}r_{c,1}'(\bar{n}_1) = 0
\end{equation}
holds. Then $\bar{n}_1$ is unique and the minimum $n_1^*$ of $u_1$ derived in Proposition~\ref{prop:obj_1D_convexity} is bounded from above independently of the budget $p$ as
\begin{equation}
n_1^* \leq \max\{1, \bar{n}_1\}\,.
\label{eq:UpperBoundN1Star}
\end{equation}
\end{proposition}
\begin{proof}
Consider $h_1$ introduced in the proof of Proposition~\ref{prop:obj_1D_n_star_boundness}. There is an $\bar{n}_1 \in (0, \infty)$ with $h_1'(\bar{n}_1) = 0$ as given by \eqref{eq:Uniqueness:ExistStat}. The stationary point $\bar{n}_1$ is unique because $h_1'$ is strictly increasing because we assumed that  \eqref{eq:ConvexityAuxiliaryAsm} holds in $(0, \infty)$. Furthermore, the stationary point $\bar{n}_1$ is independent of $p$ because $h_1$ does not depend on $p$. 

Case 1: If $\bar{n}_1 < 1$, then this implies that $h_1' > 0$ (strictly increasing because $h_1'' > 0$) for all $[1, p - 1]$ for all $p \geq 1$ and thus $u_1' > 0$ in $[1, p - 1]$. Thus, $n_1^* =1$ and \eqref{eq:UpperBoundN1Star} is a bound independent of $p$. 

Case 2: Let now $\bar{n}_1 \geq 1$. First, consider budgets $p > \bar{n}_1 + 1$ so that $\bar{n}_1 \in [1, p - 1)$. Because $h_1'(\bar{n}_1) = 0$ with $\bar{n}_1 \in [1, p - 1)$ and $h_1'$ strictly increasing, we have that $h_1' > 0$ in $(\bar{n}_1, p - 1]$ and thus $u_1' > 0$ in $(\bar{n}_1, p - 1]$, which means $u_1$ cannot be strictly decreasing on all of $[1, p - 1]$. Thus, the minimum of $u_1$ in $[1, p - 1]$ is either at a stationary point of $u_1$ or at 1. If it is a stationary point, then $h_1'(n_1^*) < 0$ has to hold to obtain $u_1'(n_1^*) = 0$ and we obtain $h_1'(n_1^*) < 0 = h_1'(\bar{n}_1)$, which implies $n_1^* \leq \bar{n}_1$ because $h_1'$ is strictly increasing and thus \eqref{eq:UpperBoundN1Star} is a $p$-independent bound of $n_1^*$. If it $n_1^* = 1$, then $1 = n_1^* \leq \bar{n}_1$ still holds because $\bar{n}_1 \geq 1$. Second, consider budgets $p \leq \bar{n}_1 + 1$ so that $\bar{n}_1 \geq p - 1$. Then, it holds $n_1^* \leq p - 1 \leq \bar{n}_1$, which again leads to the bound \eqref{eq:UpperBoundN1Star}.

This shows that depending on the properties of $h_1$, the maximum of $1$ and $\bar{n}_1$ is an upper bound of~$n_1^*$. 
\end{proof}

\subsection{Context-aware multi-fidelity Monte Carlo sampling with two or more low-fidelity models} \label{subsubsec:camfmc_multiple_models}
The second novelty of our proposed context-aware approach is that we can consider more than just one low-fidelity model.
In the following, we introduce a sequential training approach to fit hierarchies of low-fidelity models for the CA-MFMC estimator, where in each step the optimal trade off between training and sampling is achieved.

We first show an upper bound of the MSE \eqref{eq:mse_camfmc} in terms of accuracy and cost rates for the case with $k > 1$ low-fidelity models: 
\begin{lemma}\label{lm:BoundMSEMulti}
Let Assumptions \ref{assump:acc_rates} and \ref{assump:cost_rates} hold. Then, the MSE \eqref{eq:mse_camfmc} of the CA-MFMC estimator $\hat{E}^{\mathrm{CA-MFMC}}_{n_1, \dots, n_k}$ can be bounded as 
\begin{equation} \label{eq:camfmc_ND_mse_upper_bound_rates}
\mathrm{MSE}[\hat{E}^{\mathrm{CA-MFMC}}_{n_1, \dots, n_k}] \leq
\frac{(k + 1) \sigma_0^2}{p_{k-1} - n_k}
\left(\kappa_{k-1} + w_{k-1}(n_{k - 1})c_{a,k}r_{a,k}(n_k) + c_{c,k}r_{c,k}(n_k)\right)\,,
\end{equation}
where
\begin{equation}
p_{k-1} = p - \sum_{j = 1}^{k-1}n_k, \quad p_0 = p\,,\qquad \kappa_{k - 1} = \sum_{j = 0}^{k-2}w_j(n_j)(1 - \rho_{j+1}^2(n_{j+1})), \quad \kappa_0 = 0\,.
\label{eq:PK-1Kappa}
\end{equation}
\end{lemma}
\begin{proof}
With the sums of squares inequality,
\[
\left(\sum_{i = 0}^ka_i\right)^2 \leq (k + 1)\sum_{i = 0}^k a_i^2\,,\qquad a_0, \dots, a_k \in \mathbb{R}\,,
\]
we obtain
\begin{align*}
\mathrm{MSE}[\hat{E}^{\mathrm{CA-MFMC}}_{n_1, \dots, n_k}] = & \frac{\sigma_0^2}{p - \sum_{j=0}^{k} n_j}
\left(\sum_{j=0}^{k} \sqrt{w_j(n_j)\left(\rho_j^2(n_j) - \rho_{j + 1}^2(n_{j+1})\right)} \right)^2\\
\leq & \frac{\sigma_0^2(k + 1)}{p - \sum_{j=0}^{k} n_j}\sum_{j = 0}^k w_j(n_j)(\rho_j^2(n_j) - \rho_{j + 1}^2(n_{j+1}))\\
= & \frac{\sigma_0^2(k + 1)}{p - \sum_{j=0}^{k} n_j}\left(\sum_{j = 0}^{k-2} w_j(n_j)(\rho_j^2(n_j) - \rho_{j + 1}^2(n_{j+1})) + \right.\\
& \qquad \qquad \qquad \qquad \qquad  w_{k-1}(n_{k-1})(\rho_{k-1}^2(n_{k-1}) - \rho_k^2(n_k)) + w_k(n_k)\rho_k^2(n_k)\bigg)\,,
\end{align*}
where we used the convention that $\rho_{k + 1} = 0$ in the last equation. With $\rho_j^2 \leq 1$ for $j = 0, \dots, k$ and $p_{k - 1}, \kappa_{k - 1}$ defined in \eqref{eq:PK-1Kappa}, the bound \eqref{eq:camfmc_ND_mse_upper_bound_rates} follows with Assumptions~\ref{assump:acc_rates} and \ref{assump:cost_rates}.
\end{proof}

Bound \eqref{eq:camfmc_ND_mse_upper_bound_rates} in Lemma~\ref{lm:BoundMSEMulti} decomposes into the component $\kappa_{k - 1}$ which depends on $n_1, \dots, n_{k - 1}$ only and components $w_{k-1}(n_{k-1})(\rho_{k-1}^2(n_{k-1}) - \rho_k^2(n_k))$ and $w_k(n_k)\rho_k^2(n_k)$ that can be bounded with Assumptions~\ref{assump:acc_rates} and \ref{assump:cost_rates} so that they depend on $n_k$ for a fixed $n_{k-1}$.
This decomposition motivates a sequential approach of adding low-fidelity models to the CA-MFMC estimator. 
At iteration $j = 1, \dots, k$, we consider the objective
\begin{equation}\label{eq:CAMFMC:ObjJ}
u_j(n_j; n_1, \dots, n_{j-1}): [1, p_{j - 1} - 1] \to (0, \infty), \qquad n_j \mapsto \frac{1}{p_{j - 1} - n_j}\left(\kappa_{j - 1} + \hat{c}_{a,j}r_{a,j}(n_j) + c_{c,j}r_{c,j}(n_j)\right)\,,
\end{equation}
where $\hat{c}_{a,j} = w_{j-1}(n_{j-1})c_{a,j}$. 
For $j = 1$, we obtain the objective defined in \eqref{eq:CAMFMC_obj_1D} with the convention that $n_0 = 0$ and $w_0 = 1$. 
For $j > 1$, we obtain objectives $u_j$ in \eqref{eq:CAMFMC:ObjJ} that depend on $n_1, \dots, n_{j - 1}$. 
For given $n_1, \dots, n_{j-1}$, there is a global, unique minimizer of $u_j$ in $[1, p_{j - 1} - 1]$ as the following proposition shows.

\begin{proposition}\label{prop:MultiCAMFMC}
Let Assumptions~\ref{assump:acc_rates} and \ref{assump:cost_rates} hold. For $j = 1$, let Proposition~\ref{prop:obj_1D_convexity} apply and let $n_1^*$ be the corresponding unique global minimum. Iterate now over $j = 2, \dots, k$ and consider the objective $u_j(n_j; n_1^*, \dots, n_{j-1}^*)$ defined in \eqref{eq:CAMFMC:ObjJ} at iteration $j$. If it holds that 
\begin{equation} 
\hat{c}_{a,j}r_{a,j}^{\prime \prime}(n_j) + c_{c,j}r_{c,j}^{\prime \prime}(n_j) > 0\,,
\label{eq:ExtraStrictConvex}
\end{equation}
then there exists a unique minimizer $n_j^* \in [1, p_{j-1} - 1]$ of $u_j$. 
\end{proposition}
\begin{proof}
The function $n_j \mapsto \kappa_{j - 1} + \hat{c}_{a,j}r_{a,j}(n_j) + c_{c,j}r_{c,j}(n_j)$ satisfies \eqref{eq:ConvexityAuxiliaryAsm} in $[1, p_{j - 1}-1]$ because $\kappa_{j-1}$ is constant in $n_j$ and \eqref{eq:ExtraStrictConvex} is assumed. Thus, the same arguments as in the proof of Proposition~\ref{prop:obj_1D_convexity} apply and thus unique minimizers $n_2^*, \dots, n_k^*$ exist.
\end{proof}

The next proposition shows the minimizers of \eqref{eq:CAMFMC:ObjJ} for $j = 1, \dots, k$ are bounded from above, independent of the computational budget $p$.
\begin{proposition}\label{prop:BoundMulti}
Let Proposition~\ref{prop:MultiCAMFMC} apply, with the condition \eqref{eq:ExtraStrictConvex} holding in $(0, \infty)$. Furthermore, analog to Proposition~\ref{prop:obj_1D_n_star_boundness}, let there exist $\bar{n}_j \in (0, \infty)$ so that
\[
\hat{c}_{a,j}r_{a,j}^{\prime}(\bar{n}_j) + c_{c,j}r_{c,j}^{\prime}(\bar{n}_j) = 0
\]
holds for all $j = 1, \dots, k$. Then, for $j = 1, \dots, k$, there exist bounds $\bar{n}_1^*, \dots, \bar{n}_k^*$ that are independent of $p$ so that $n_j^* \leq \bar{n}_j^*$ for $j = 1, \dots, k$. 
\end{proposition}
\begin{proof}
Proof via induction and applying Proposition~\ref{prop:obj_1D_n_star_boundness} sequentially to Proposition~\ref{prop:MultiCAMFMC}. 
\end{proof}

Propositions~\ref{prop:obj_1D_convexity}--\ref{prop:BoundMulti} highlights that even though using more than $n_{\ell}^*$ high-fidelity model evaluations as training data to construct the $\ell$th low-fidelity model can lead to a more accurate low-fidelity model in terms of correlation coefficient, it also increases its evaluation costs, which ultimately leads to an increase of the upper bound of the MSE of the MFMC estimator \eqref{eq:camfmc_ND_mse_upper_bound_rates} for a fixed budget and thus a poorer estimator.
This shows that it is beneficial in multi-fidelity methods to trade off accuracy and evaluation costs of the low-fidelity models.

\begin{rmk} \label{rmk:model_swap}
Increasing the model hierarchy by adding the $j$th context-aware low-fidelity model $f^{(j)}_{n_j}$ must ensure that the MFMC ordering \eqref{eq:MFMC_ordering} is satisfied, i.e., the accuracy and evaluation cost rates corresponding to $f^{(j)}_{n_j}$, evaluated at $n_j^*$, have to satisfy \eqref{eq:MFMC_ordering}. 
If \eqref{eq:MFMC_ordering} is not satisfied, the order of the low-fidelity models in the multi-fidelity hierarchy must be changed accordingly.
\end{rmk}
\section{Numerical experiments and discussion} \label{sec:results}
We now present numerical results.
In Section~\ref{subsec:test_case_1}, we consider a heat conduction problem given by a parametric elliptic partial differential equation (PDE) with nine uncertain parameters, defined on a two-dimensional spatial domain.
Our goal in this experiment is to draw initial insights about the proposed context-aware learning algorithm. 
To this end, we will employ two heterogeneous low-fidelity models: an accurate low-fidelity model that is computationally more expensive than a second, less accurate, low-fidelity model.
Section~\ref{subsec:test_case_2} considers plasma micro-turbulence simulations that depend on $12$ uncertain inputs.
In this example, we consider three low-fidelity models: a coarse-grid model and two non-intrusive data-driven low-fidelity models based on sparse grid approximations and deep neural network regression. 

\subsection{Heat conduction in a two-dimensional spatial domain} \label{subsec:test_case_1}
All numerical experiments in this section were performed using an eight core Intel i7-10510U CPU and 16 GB of RAM.

\subsubsection{Setup}
The thermal block problem \cite{RozzaPateraSurvey} is defined on a two-dimensional spatial domain $\Omega = (0, 1)^2 = \bigcup_{i = 1}^{d} \Omega_i$, divided into $d = B_1 \times B_2$ non-overlapping vertical and horizontal square blocks, $\Omega_i$ with $i = 1, \ldots d$.
The mathematical model is given by the parametric elliptic PDE
\begin{equation} \label{eq:TB}
\begin{split}
- \operatorname{div}  k(x, y, \boldsymbol{\theta}) \nabla u(x, y) = 0 & \text{ in } \Omega\,, \\
u(x, y) = 0 & \text{ on } \Gamma_D\,, \\
k(x, y, \boldsymbol{\theta}) \nabla u(x, y) \cdot \boldsymbol{n} = j & \text{ on } \Gamma_{N, j}, \quad j = 0, 1,
\end{split}
\end{equation}
where $\Gamma_D$ is a Dirichlet boundary, $\Gamma_{N, 0}$ and $\Gamma_{N, 1}$ are Neumann boundaries, and
\begin{equation*}
k(x, y, \boldsymbol{\theta})  = \sum_{i = 1}^{d} \theta_i \chi_{\Omega_i}(x, y)
\end{equation*}  
is the piece-wise constant heat conductivity field parametrized by $\boldsymbol{\theta} = [\theta_1, \theta_2, \ldots, \theta_{d}] \subset \mathbb{R}^d$, where $\chi_{\Omega_i}$ is the indicator function of $\Omega_i$. We set $B_1 = B_2 = 3$ and thus $d = 9$. 
Here, $k(x, y, \boldsymbol{\theta})$ is parametrized by $d = 9$ uniformly distributed random parameters in $[1, 10]^9$.
The output of interest is the mean heat flow at the Neumann boundary $\Gamma_{N, 1}$ given by
\begin{equation} \label{eq:TB_ooi}
\mathbb{E}[u(\boldsymbol{\theta})] = \int_{\Gamma_{N, 1}} u(x, y, \boldsymbol{\theta}) \mathrm{d} x \mathrm{d} y.
\end{equation}
This setup is a slight modification of the problem considered by Patera and Rozza in \cite{RozzaPateraSurvey}.

The high-fidelity model is a finite element discretization of \eqref{eq:TB} consisting of $7,200$ triangular elements with mesh width $h = \sqrt{2}/60$, provided by the \textsc{RBMatlab}\footnote{\texttt{https://www.morepas.org/software/rbmatlab/}}.
The high-fidelity model evaluation cost is $0.1150$ seconds.
Moreover, its variance, estimated using $100$ MC samples, is $\sigma_0^2 \approx 0.0018$.  

\subsubsection{A low-fidelity model based on the reduced basis method}
The reduced-basis (RB) low-fidelity model $f^{(1)}_{n_1}$ is constructed using a greedy strategy, as described, for example, in \cite{Bi11}. 
We employ the implementation provided by \textsc{RBMatlab}.
It has been shown that greedy RB low-fidelity models have exponential accuracy rates for problems similar to the thermal block \cite{Bi11}, which is what we use when fitting the error decay.
Once the reduced basis is found in the offline stage, evaluating the low-fidelity model online entails solving a dense linear system of size equal to the number of reduced bases to find the reduced basis coefficients. 
Therefore, we model the evaluation cost rate as algebraic in the reduced-model dimension.

The constants in the rate functions are estimated via regression from pilot runs.
To estimate the constants in the exponential accuracy rate, we use $100$ high- and low-fidelity evaluations .
For estimating the constants in the evaluation cost rate, we average the runtimes of the low-fidelity model constructed using increasing values of $n_1$, evaluated at $1,000$ MC samples.
We perform these runtime measurements $50$ times and average the results.
The estimated rates are visualized in Figure~\ref{fig:TB_RB_rates} and the constants are shown in Table~\ref{tab:TB_RB_rate_constants}.
\begin{figure}[htbp]
\centering
\includegraphics[width=0.8\textwidth]{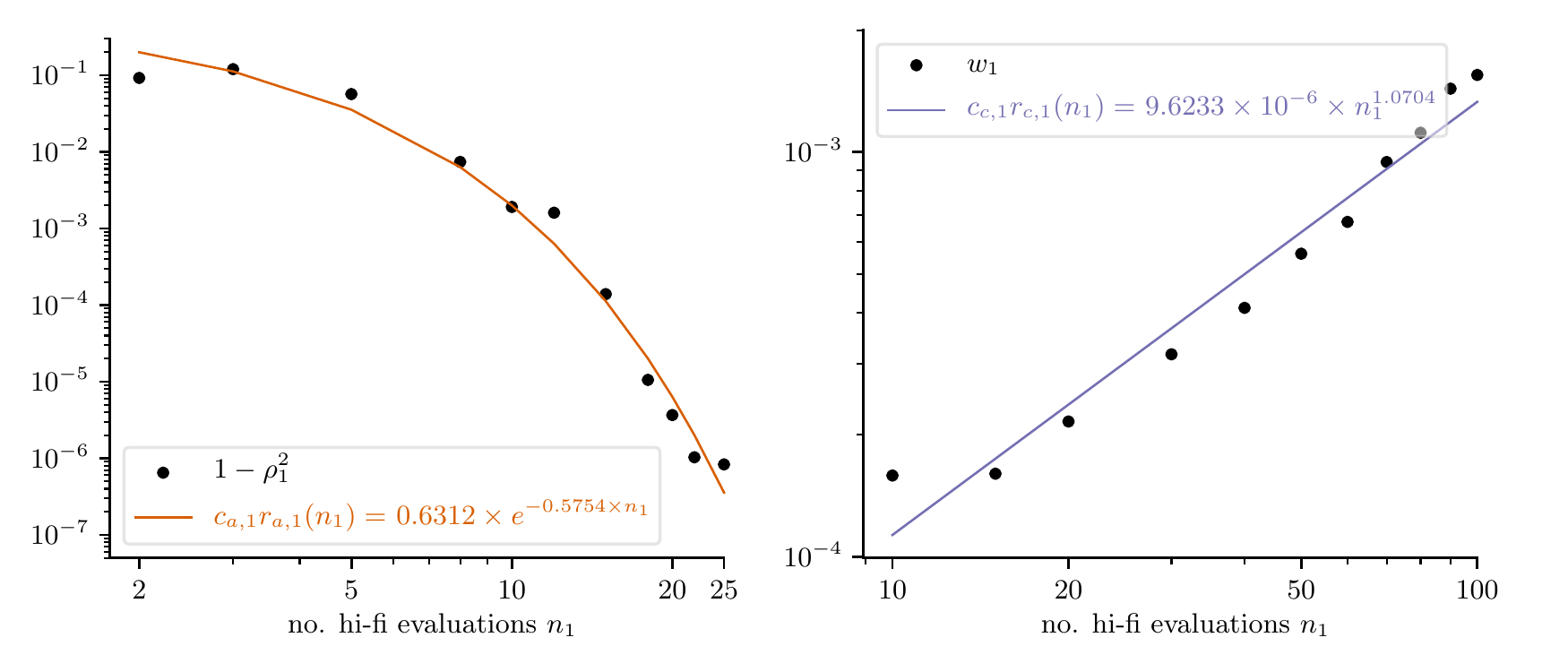}
\caption{Thermal block: Estimated accuracy (left) and evaluation cost rates (right) for the RB low-fidelity model.}
\label{fig:TB_RB_rates}
\end{figure}
\begin{table}[htbp]
\centering
\begin{tabular}{cc}
\begin{tabular}{|c|c|c|}
    \hline
    rate & rate type & estimated rate constants\bigstrut[t] \\
    \hline
    accuracy: $c_{a,1}r_{a,1}(n_1)$ & exponential: $r_{a,1}(n_1) = e^{-\alpha_1 n_1}$ & $c_{a,1} = 0.6312$, $\alpha_1 = 0.5754$ \bigstrut[t]  \\
    \hline
    evaluation cost: $c_{c,1}r_{c,1}(n_1)$ & algebraic: $r_{c,1}(n_1) = n_1^{\beta_1}$ & $c_{c,1} = 9.6233 \times 10^{-5}$, $\beta_1 = 1.0704$ \bigstrut[t]\\
    \hline
\end{tabular} &
\end{tabular}
\caption{Thermal block: Estimated accuracy and evaluation cost rates and their constants for the RB low-fidelity model.}
\label{tab:TB_RB_rate_constants}
\end{table}

\subsubsection{A data-fit low-fidelity model}
The second low-fidelity model  $f^{(2)}_{n_2}$ that we consider in the thermal block example is an $\varepsilon$-support vector regression ($\varepsilon$-SVR) model.
Our numerical implementation is based on \textsc{libsvm} \cite{CLin11}.
The training data consists of pairs of pseudo-random realizations of the nine-dimensional input and the corresponding high-fidelity model evaluations.
In our experiments, we used $\varepsilon = 10^{-2}$.
We model the accuracy and evaluation cost rates as algebraic.
To estimate the constants in the accuracy rate, we use $100$ high- and low-fidelity evaluations.
The constants in the evaluation cost rate are estimated by averaging the runtimes of the low-fidelity model evaluated at $1,000$ MC samples.
We perform these runtime measurements $50$ times and average the results.
The estimated rates are visualized in Figure~\ref{fig:TB_SVR_rates} and the constants are shown in Table~\ref{tab:TB_SVR_rate_constants}. 
Note that the $\varepsilon$-SVR low-fidelity model is less accurate but cheaper to evaluate than the RB low-fidelity model.
\begin{figure}[htbp]
\centering
\includegraphics[width=0.8\textwidth]{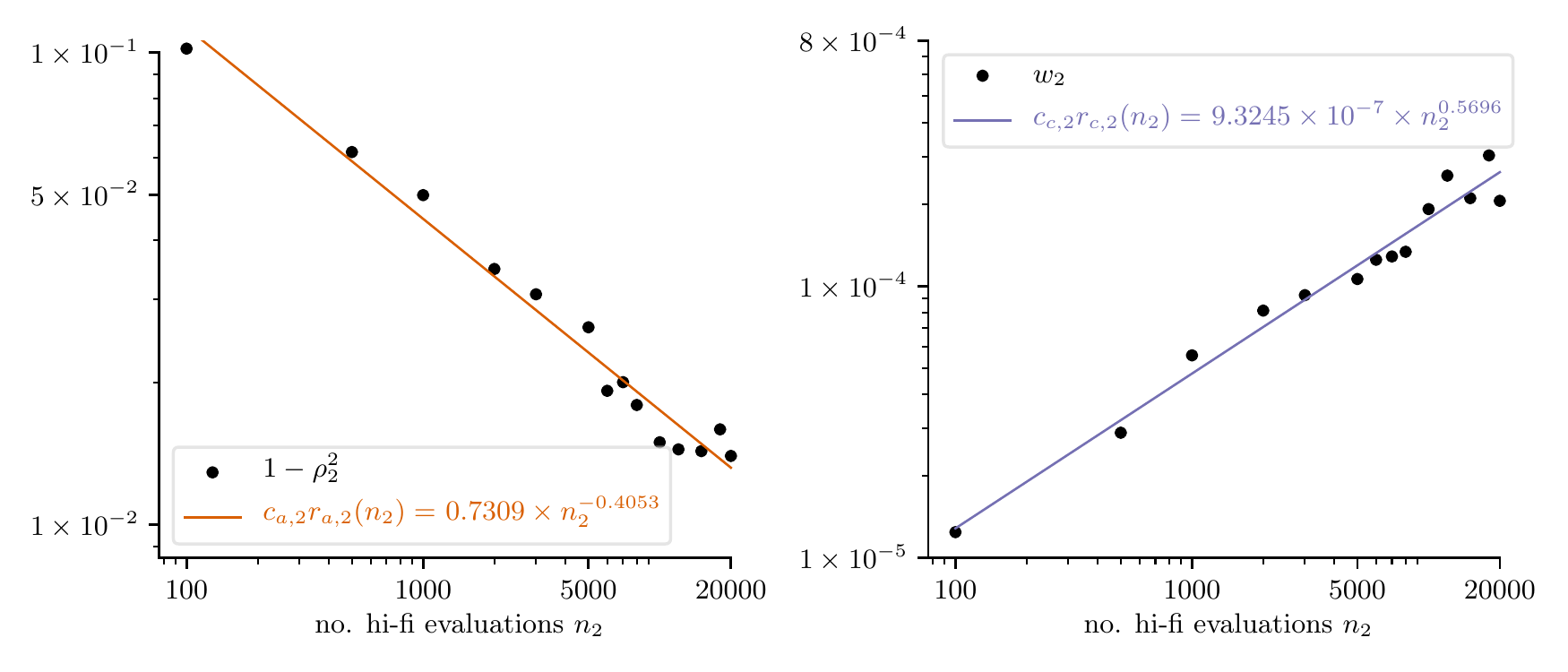}
\caption{Thermal block: Estimated accuracy (left) and evaluation cost rates (right) for the $\varepsilon$-SVR low-fidelity model.}
\label{fig:TB_SVR_rates}
\end{figure}
\begin{table}[htbp]
\centering
\begin{tabular}{cc}
\begin{tabular}{|c|c|c|}
    \hline
    rate & rate type & estimated rate constants\bigstrut[t] \\
    \hline
    accuracy: $c_{a,2}r_{a,2}(n_2)$ & algebraic: $r_{a,2}(n_2) = n_2^{-\alpha_2}$ & $c_{a,2} = 0.7309$, $\alpha_2 = 0.4053$  \bigstrut[t]  \\
    \hline
    evaluation cost: $c_{c,2}r_{c,2}(n_2)$ & algebraic: $r_{c,2}(n_2) = n_2^{\beta_2}$ & $c_{c,2} = 9.3245 \times 10^{-7}$, $\beta_2 = 0.5696$ \bigstrut[t]\\
    \hline
\end{tabular} &
\end{tabular}
\caption{Thermal block: Accuracy and evaluation cost rates and their constants for the $\varepsilon$-SVR low-fidelity model.}
\label{tab:TB_SVR_rate_constants}
\end{table}

\subsubsection{Context-aware multi-fidelity Monte Carlo with RB low-fidelity model} \label{subsubsec:TB_one_lo_fi_model}
We first consider only the RB low-fidelity model and budgets $p \in \{5, 10, 30, 50, 80, 100, 300, 500 \}$ seconds.
We compare, in terms of MSE, our context-aware estimator with standard MC sampling and MFMC in which the RB low-fidelity model is constructed using a fixed, a priori chosen number of basis independent of the budget $p$. 
We compute the MSEs from $N$ replicates as
\begin{equation}\label{eq:mse_replicates}
e_{\text{MSE}}(\hat{E}^{(\cdot)}) = \frac{1}{N} \sum_{n=1}^{N} \left(\hat{\mu}_{\text{ref}} - \hat{E}^{(\cdot)}\right)^2,
\end{equation}
where $\hat{\mu}_{\text{ref}}$ represents the reference mean estimator and $\hat{E}^{(\cdot)}$ is either the CA-MFMC, standard MFMC or the MC estimator.
The reference result $\hat{\mu}_{\text{ref}}$ was obtained using the context-aware multi-fidelity estimator in which the two low-fidelity models, RB and $\varepsilon$-SVR, were sequentially added with a computational budget $p_{\text{ref}} = 10^3$ seconds (cf.~Section~\ref{subsubsec:TB_sequentially_adding_second_lo_fi_model}).
To distinguish between the employed estimators, we use ``\emph{std.~MC: } $f^{(0)}$'' to denote the standard MC estimator, ``\emph{MFMC: } $f^{(0)}, f^{(1)}$'' to refer to the MFMC estimator in which the RB low-fidelity model is statically constructed independent of the budget $p$, and ``\emph{CA-MFMC: } $f^{(0)}, f^{(1)}_{n_1}$'' to refer to the context-aware multi-fidelity estimator with the RB low-fidelity model.

Based on the estimated rate parameters in Table~\ref{tab:TB_RB_rate_constants}, both the accuracy and cost rates of the RB low-fidelity model are strictly convex for all $n_1 \in [1, p  - 1]$ and any budget $p > 0$, and hence Propositions~\ref{prop:obj_1D_convexity} and \ref{prop:obj_1D_n_star_boundness} apply. 
Therefore the optimal, context-aware number of high-fidelity model evaluations for constructing the RB low-fidelity model, $n_1^*$, exists, is unique and bounded with respect to the budget. 
The optimal number of training samples $n_1^*$ for increasing budgets is shown in Figure~\ref{fig:TB_RB_n_star_offline_online}, on the left.
Notice that $n_1^* \leq 18$ and this upper bound is attained for budgets $p \geq 10$ seconds, which is because the RB low-fidelity models is accurate already with few basis vectors.
Figure~\ref{fig:TB_RB_n_star_offline_online}, right, depicts the split of the total computational budget between constructing the RB low-fidelity model (offline phase, using $n_1^*$ high-fidelity model evaluations) and multi-fidelity sampling (online phase, using the remaining budget). 
Because $n_1^*$ is bounded, the percentage of the budget allocated to constructing the RB low-fidelity model decreases with $p$. 
\begin{figure}[bt]
\centering
\includegraphics[width=1.0\textwidth]{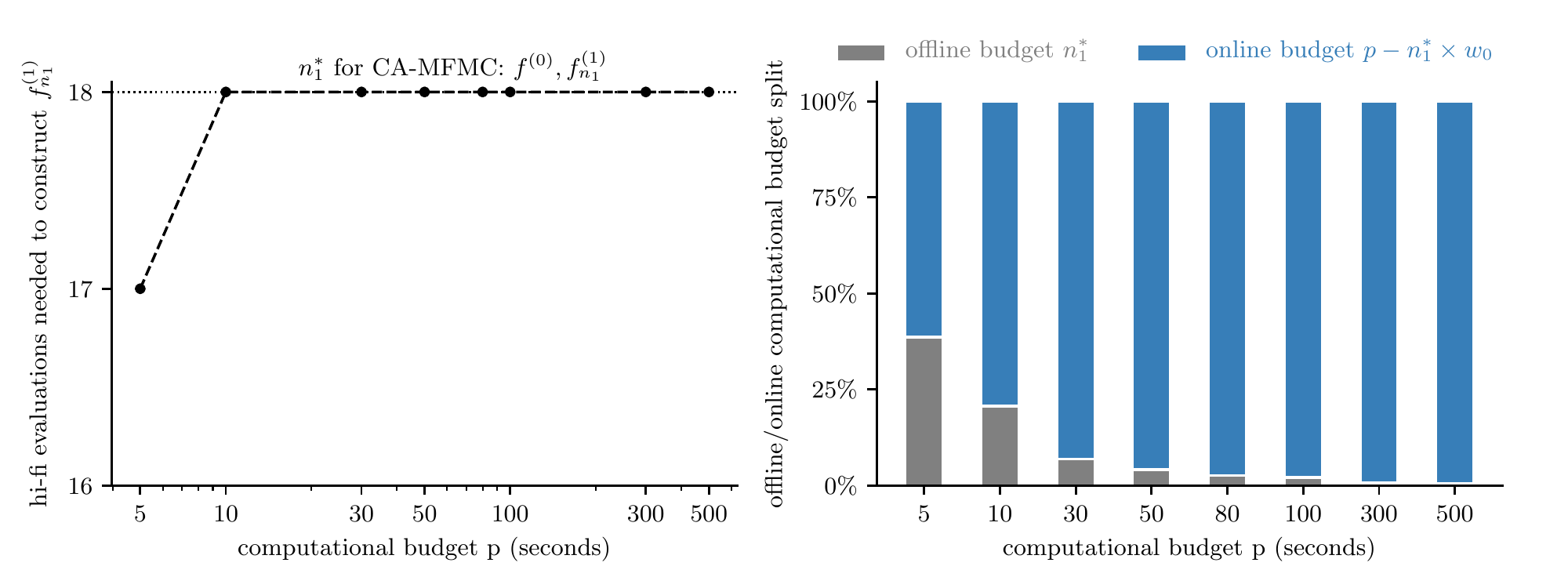}
\caption{Thermal block: Left, optimal number of high-fidelity training samples for training the RB low-fidelity model for the context-aware estimator. Right, visualization of how the context-aware estimator splits the computational budget between training (offline) and sampling (online).}
\label{fig:TB_RB_n_star_offline_online}
\end{figure}

Figure~\ref{fig:TB_RB_MSE} compares the MSEs of regular MC, static MFMC with RB low-fidelity models of fixed dimension $2, 8, 50$, and CA-MFMC estimators.
The left plot reports the analytical MSEs derived using equations \eqref{eq:mse_mfmc}, \eqref{eq:mse_mc} and \eqref{eq:mse_camfmc}, which can be evaluated with the constants reported in Tables~\ref{tab:TB_RB_rate_constants} and \ref{tab:TB_RB_MFMC}. 
\begin{table}[htbp]
\centering
\begin{tabular}{cc}
\begin{tabular}{|c|c|c|c|}
    \hline
    $n_1$ & $\rho_1$ & $w_1$ & $\sigma_1^2$\bigstrut[t] \\
    \hline
    $2$ & $0.9277$ & $1.0915 \times 10^{-4}$ & $9.4347 \times 10^{-4}$ \bigstrut[t]  \\
    \hline
    $8$ & $0.9939$ & $1.9791 \times 10^{-4}$ & $0.0015$ \bigstrut[t]\\
    \hline
    $50$ & $0.9999$ & $6.7539 \times 10^{-4}$ & $0.0018$ \bigstrut[t]\\
    \hline
\end{tabular} &
\end{tabular}
\caption{Thermal block: Correlations, evaluation costs and variances of the RB low-fidelity model $f^{(1)}$ with dimensions $2, 8, 50$, respectively.}
\label{tab:TB_RB_MFMC}
\end{table}
Computing the analytical MSEs requires no numerical simulations of the forward model as long as the constants in Tables~\ref{tab:TB_RB_rate_constants} and \ref{tab:TB_RB_MFMC} are available.
The right plot shows the MSEs obtained numerically by taking $N = 50$ replicates of the estimators.
All multi-fidelity estimators give a lower MSE than standard MC sampling by at least half an order of magnitude, even when using only two dimensions in the static RB low-fidelity model. 
Overall, the proposed context-aware multi-fidelity estimator gives the lowest MSE of all employed estimators. 
It is about four orders of magnitude more accurate than standard MC sampling in terms of MSE. 
Our context-aware estimator is also more accurate than the MFMC estimators with static models because $2$ and $8$ are too few basis vectors and $50$ are too many basis vectors that make the low-fidelity model unnecessarily expensive to evaluate. 
In contrast, our context-aware estimator optimally trades off training and sampling and so achieves the lowest MSE of all estimators, which agrees with Proposition~\ref{prop:obj_1D_convexity} and Proposition~\ref{prop:obj_1D_n_star_boundness}. 
Even though the static MFMC estimator with $n_1 = 50$ basis vectors is close to the CA-MFMC estimator, it is unclear how to choose the dimension in static MFMC a priori and if, e.g., $50$ is a good choice, whereas the proposed context-aware approach provides the optimal $n_1^*$ with respect to an upper bound of the MSE for a given budget $p$; cf.~Section~\ref{subsubsec:camfmc_one_model_setup} and Figure~\ref{fig:TB_RB_n_star_offline_online}, left.
\begin{figure}[bt]
  \centering
  \includegraphics[width=1.0\textwidth]{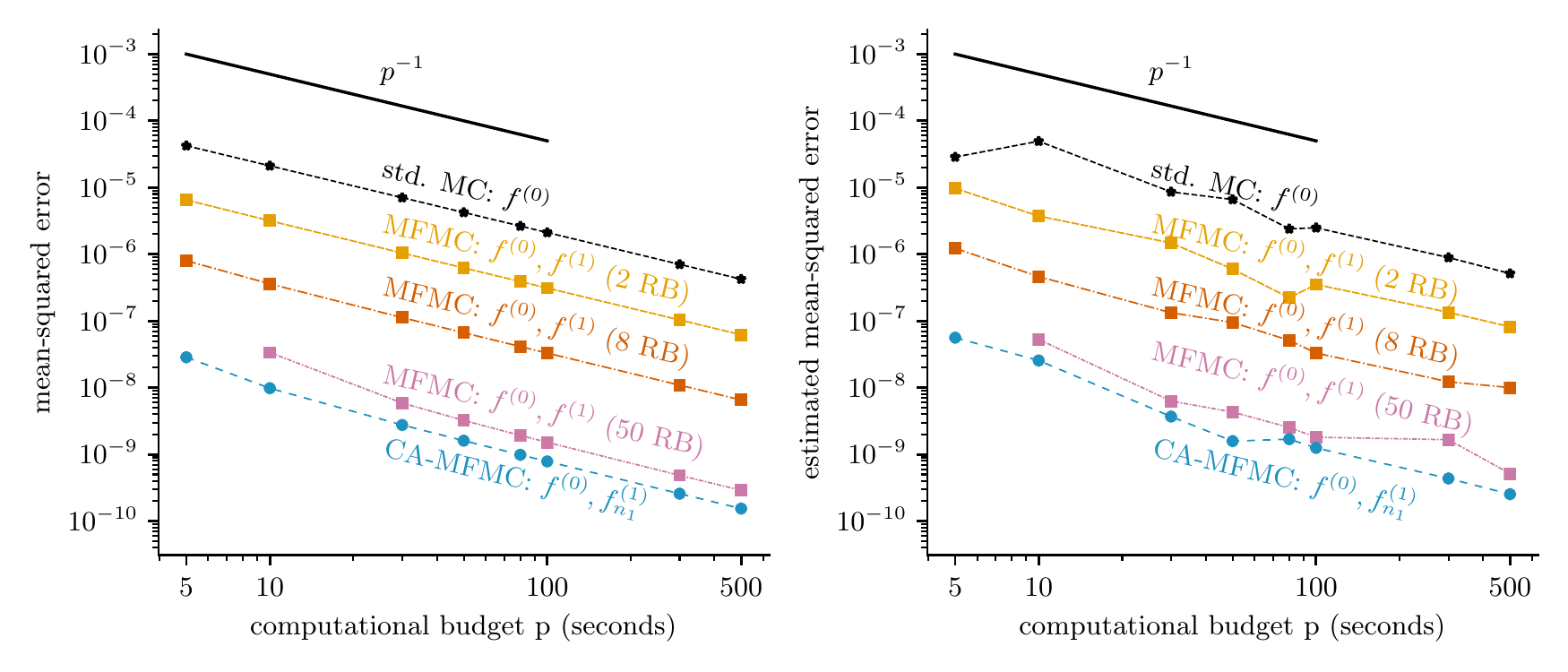}
  \caption{Thermal block: Analytical MSE is shown left and estimated MSE is shown right. The context-aware MFMC estimator CA-MFMC achieves the lowest MSE by trading off the costs of training the low-fidelity model and sampling. Even though we see that the static MFMC with an RB model of dimension $50$ performs well here, it is unclear before the error has been computed what dimension of the RB low-fidelity model is a good choice. Furthermore, the static estimator with RB model of dimension $50$ requires at least $50$ high-fidelity model evaluations to start with, whereas the CA-MFMC adaptively chooses the number of high-fidelity model evaluations $n_1^*$ for training the low-fidelity model based on the available budget.}
  \label{fig:TB_RB_MSE}
\end{figure}

\subsubsection{Context-aware multi-fidelity Monte Carlo with RB and data-fit low-fidelity models} \label{subsubsec:TB_sequentially_adding_second_lo_fi_model}
We now sequentially add, as detailed in Section~\ref{subsubsec:camfmc_multiple_models}, the second low-fidelity model that is based on $\varepsilon$-SVR.
We consider budgets of $p \in \{5, 10, 30, 50\}$ seconds. 
The abbreviation ``\emph{CA-MFMC: } $f^{(0)}, f^{(1)}_{n_1}, f^{(2)}_{n_2}$'' refers to the CA-MFMC estimator in which the $\varepsilon$-SVR low-fidelity model $f^{(2)}_{n_2}$ is sequentially added after the RB low-fidelity model $f^{(1)}_{n_1}$.

As it can be seen from Table~\ref{tab:TB_SVR_rate_constants}, the algebraic cost rate of the $\varepsilon$-SVR low-fidelity model is concave since $0 < \beta_2 < 1$, which implies that we need to verify whether the assumptions in Propositions~\ref{prop:MultiCAMFMC} 
hold true.
We therefore verify whether the function $h_2(n_2) = \kappa_{1} + \hat{c}_{a,2}r_{a,2}(n_2) + c_{c,2}r_{c,2}(n_2)$ defined in \eqref{eq:PK-1Kappa}, with $\kappa_{1} = c_{c,1} r_{c, 1}(n_1^*)$ and $\hat{c}_{a,3} = \kappa_2 c_{a, 3}$ with $\kappa_2 = c_{a,1} r_{a, 1}(n_1^*)$, satisfies condition \eqref{eq:ExtraStrictConvex}.
Since $\kappa_{1}$ and $\kappa_{2}$ are constants for a fixed $n_1^*$, $h_2(n_2)$ is independent of the budget and it is hence sufficient to verify \eqref{eq:ExtraStrictConvex} for the largest considered budget, $p = 50$ seconds.
For this budget, $n_1^* = 18$ (cf.~Figure~\ref{fig:TB_RB_n_star_offline_online}, left), and $\kappa_{1} = 2.0061 \times 10^{-5}$ and $\kappa_2 = 2.1227 \times 10^{-4}$.
Moreover, $n_2 \in [1, \lfloor p/w0 \rfloor - n_1^* - 1] = [1, 415]$.
The left plot in Figure~\ref{fig:TB_u_h_2_SVR} shows that $h_2^{\prime \prime}(n_2) > 0$ for all $n_2 \in [1, 415]$ which implies that \eqref{eq:ExtraStrictConvex} is satisfied for all budgets $0 < p \leq 50$.
Therefore the number of high-fidelity model evaluations used to construct the $\varepsilon$-SVR low-fidelity model $n_2^*$ exists and is unique 
with respect to the budget $p - n_1^* \times w_0$ that is left after adding the context-aware RB low-fidelity model.
The right plot in Figure~\ref{fig:TB_u_h_2_SVR} depicts the objective $u_2(n_2)$ defined in \eqref{eq:CAMFMC:ObjJ}. 
\begin{figure}[htbp]
\centering
\includegraphics[width=1.0\textwidth]{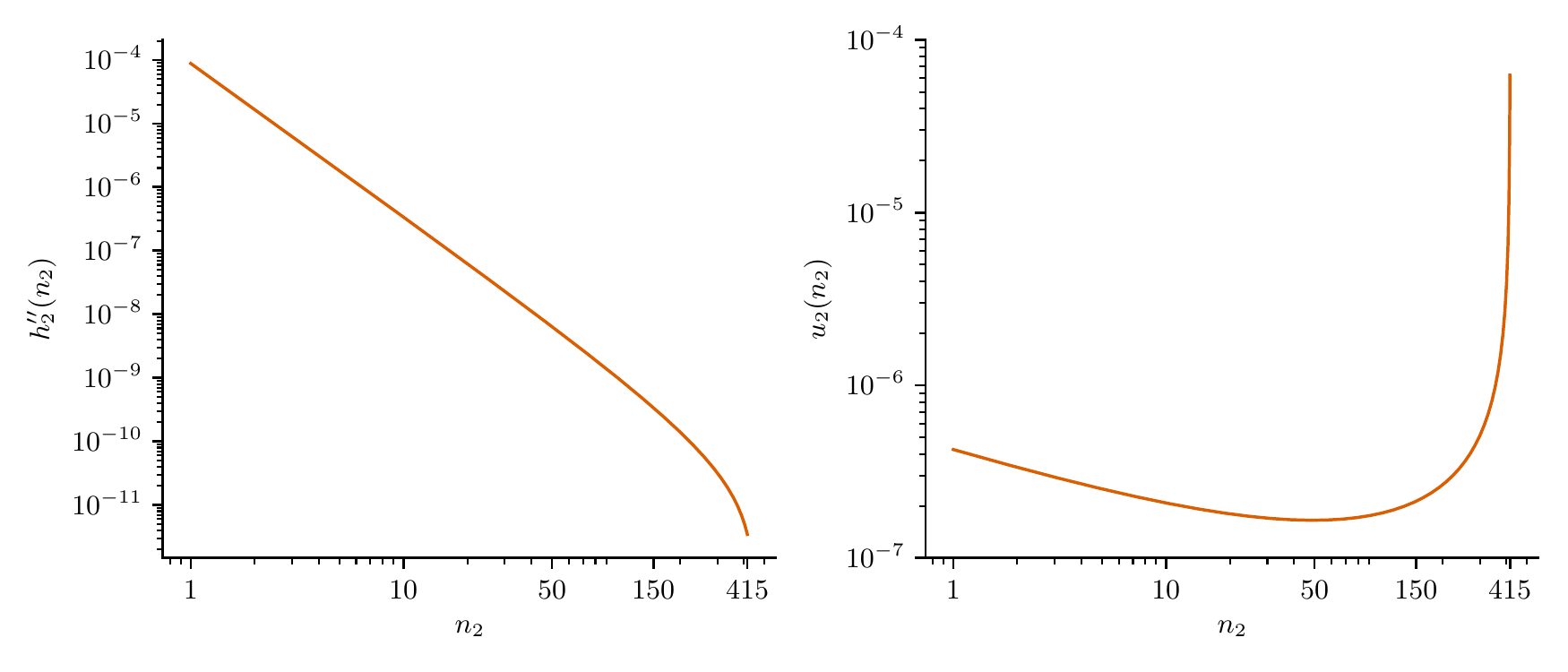}
\caption{Thermal block: The plot on the left demonstrates that 
\eqref{eq:ExtraStrictConvex} holds for budget $p = 50$ seconds, which corresponds to $n_2 \leq 415$. 
The plot on the right shows the objective $u_2$ defined in \eqref{eq:CAMFMC:ObjJ}, which has a unique, global minimizer in $[1, 415]$.} 
\label{fig:TB_u_h_2_SVR}
\end{figure}

Figure~\ref{fig:TB_RB_SVR_MSE} compares the MSE of the estimators. 
In the left plot, we show the analytical MSEs computed using equations \eqref{eq:mse_mfmc}, \eqref{eq:mse_mc} and \eqref{eq:mse_camfmc} and the constants reported in Tables~\ref{tab:TB_RB_rate_constants}, \ref{tab:TB_SVR_rate_constants} and \ref{tab:TB_RB_MFMC}, without requiring any new numerical simulations.
The MSEs reported in the right plot were computed using $N = 50$ replicates.
Sequentially adding the $\varepsilon$-SVR low-fidelity model to the context-aware estimator further decreases the MSE.
This shows that even though the slowly decreasing accuracy rate of the $\varepsilon$-SVR low-fidelity model indicates that it would necessitate a large training set to be accurate in single-fidelity settings for predicting high-fidelity model outputs, training it using at most $n_2^* = 50$ high-fidelity training samples in our context-aware learning approach is sufficient to achieve variance reduction and thus a more accurate multi-fidelity estimator.
This indicates that even low-fidelity models with poor predictive capabilities can be useful for multi-fidelity methods as long as they capture the trend of the high-fidelity model. 
\begin{figure}[htbp]
\centering
  \includegraphics[width=1.0\textwidth]{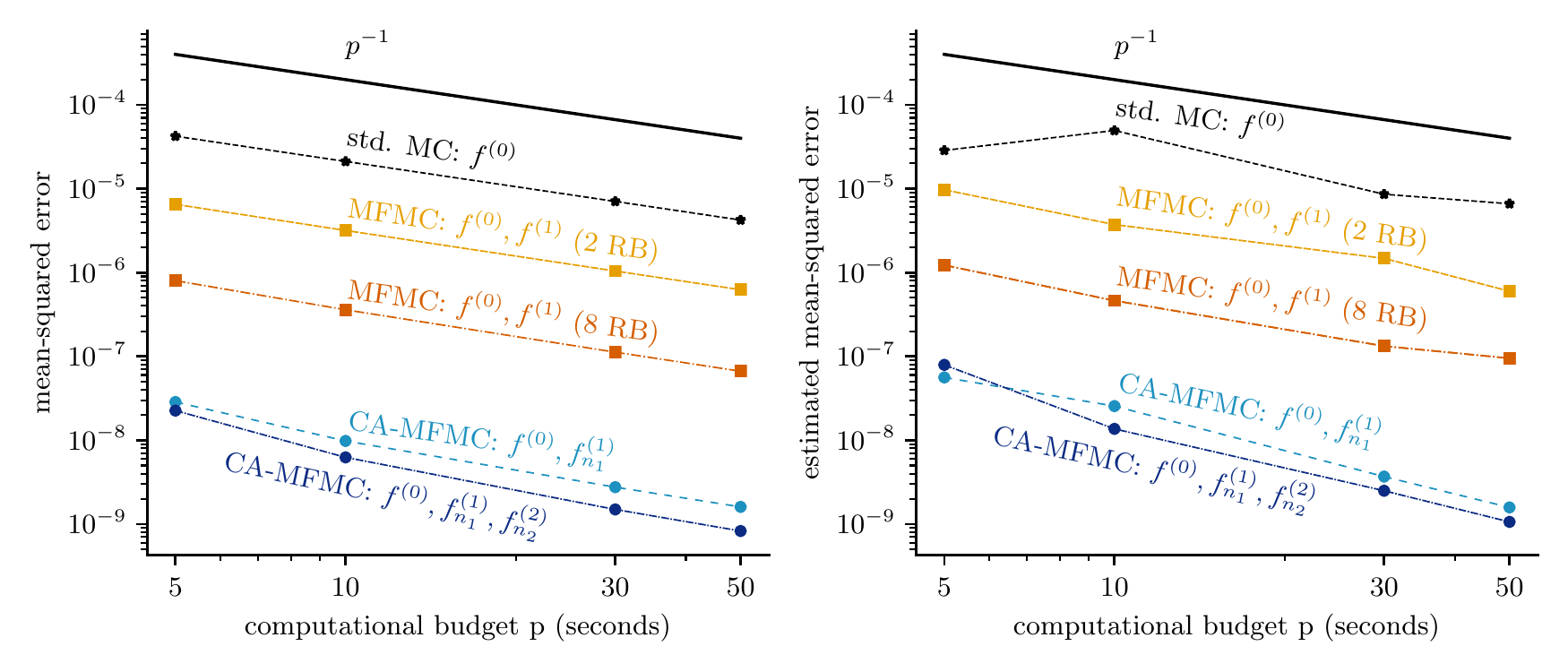}
  \caption{Thermal block: The left plot shows the analytic MSE and the right plot showed the estimated MSE. When adding the second model $f_{n_2}^{(2)}$, the MSE of the CA-MFMC estimator is further reduced compared to the CA-MFMC estimator with only one low-fidelity model.}
\label{fig:TB_RB_SVR_MSE}
\end{figure}

\subsection{Plasma micro-turbulence simulation in the ASDEX Upgrade Experiment} \label{subsec:test_case_2}

We now consider uncertainty quantification in simulations of plasma micro-turbulence in magnetic confinement devices. 
Of interest are small-scale fluctuations which cause energy loss rates despite sophisticated plasma confinements via strong and shaped magnetic fields.
The micro-turbulence is driven by the free energy provided by the steep plasma temperature and density gradients. 
However, the measurements of these gradients, as well as further decisive physics parameters affecting the underlying micro-instabilities are subject to uncertainties, which need to be quantified. 

\subsubsection{Setup of the experiment}
We focus on linear (in the phase space variables) gyrokinetic simulations (five-dimensional phase space characterized by three positions and two velocities), which are used to characterize the underlying micro-instabilities; we refer to \cite{BH07} and the references therein for details.
A parameter set from a discharge from the ASDEX Upgrade Experiment is considered, which is similar to the parameter set used in \cite{Fr18} for a validation study of gyrokinetic simulations.
The experiment considered here is characterized by two particle species, (deuterium) ions and electrons.
Moreover, the magnetic geometry is described by an analytical Miller equilibrium \cite{Miller_PoP98}.
We consider both electromagnetic effects and collisions.
Collisions are modeled by a linearized Landau-Boltzmann operator. 
The total number of uncertain inputs is $12$, which are modeled as given in Table~\ref{tab:aug_12D}. 
\begin{table}[bt]
\centering
\begin{tabular}{|c|c|c|cc|}
\hline
$\boldsymbol{\theta}$ & parameter name & symbol & left bound  & right bound \\ 
\hline
$\theta_1$ & plasma $\beta$ & $\beta$ & $0.4889 \times 10^{-3}$ & $0.5975 \times 10^{-3}$ \\
$\theta_2$ & collision frequency & $\nu_c$ & $0.6412 \times 10 ^{-2}$ & $0.8676 \times 10^{-2}$ \\
$\theta_3$ & ion and electron density gradient & $\omega_n$ & $1.1563$ & $1.9272$ \\
$\theta_4$ & ion temperature gradient & $\omega_{T_i}$ & $2.0966$ & $3.4943$ \\
$\theta_5$ & temperature ratio & $T_i/T_e$ & $0.6142$ &  $0.6788$ \\
$\theta_6$ & electron temperature gradient & $\omega_{T_e}$ & $4.0403$ & $6.7338$ \\
$\theta_7$ & effective ion charge & $Z_{\mathrm{eff}}$ & $1.2800$ & $1.9200$ \\
$\theta_8$ & safety factor & $q$ & $2.1708$ & $2.3993$ \\
$\theta_9$ & magnetic shear & $\hat{s}$ & $1.9927$ & $2.4356$ \\
$\theta_{10}$ & elongation & $k$ & $1.2815$ & $1.4165$ \\
$\theta_{11}$ & elongation gradient & $s_k$ & $0.2081$ & $0.2544$ \\
$\theta_{12}$ & triangularity & $\delta$ & $0.0528$ & $0.0583$ \\
\hline
\end{tabular}
\caption{ASDEX Upgrade Experiment: The $12$ uniformly distributed uncertain inputs.
\label{tab:aug_12D}}
\end{table}
    
We perform experiments at one bi-normal wave-number $k_y \rho_s = 0.8$.
The output of interest is the growth rate (spectral abscissa) of the dominant micro-instability mode, which corresponds to the maximum real part over the spectrum.
For this parameter set, the dominant micro-instability mode at $k_y \rho_s = 0.8$ and at the nominal values of the input parameters is electron temperature gradient-driven, characterized by a negative frequency (imaginary part), while the first subdominant mode is ion temperature gradient-driven, characterized by a positive frequency. 
However, for certain values of the ion temperature and density gradients within the bounds showed in Table~\ref{tab:aug_12D}, the first subdominant ion mode can transition to become the dominant mode.
This mode transition can be sharp and in turn lead to sharp transitions or even discontinuities in the growth rate.
To avoid the potentially detrimental effects of this discontinuity, in our experiments, we determine the growth rate of the underlying electron-driven micro-instability mode.
To this end, we compute the first two micro-instability eigenmodes and then select the growth rate that has a corresponding negative frequency.

\subsubsection{High-fidelity model}
The high-fidelity model is provided by the gyrokinetic simulation code \textsc{Gene} \cite{Je00} in the flux-tube limit \cite{Dannert_2005}. 
This assumes periodic boundary conditions over the radial $x$ and bi-normal $ky$ directions.
To discretize the five dimensional gyrokinetic phase-space domain, we use $15$ Fourier modes in the radial direction and $24$ points along the field line, in the $z$ direction.
The binormal coordinate $ky$ is decoupled in linear flux-tube simulations and hence only one such mode is used here.
In velocity space, we employ $48$ equidistant and symmetric parallel velocity grid points and $16$ Gauss-Laguerre distributed magnetic moment points. 
This results in a total of $276,480 = 15 \times 1 \times 24 \times 48 \times 16$ degrees of freedom in the $5$D phase space.

The average high-fidelity model runtime is $410.9941$ seconds on $32$ cores on one node on the Lonestar6 supercomputer at the Texas Advanced Computing Center; cf.~Section~\ref{sec:Intro}. One such node comprises two AMD EPYC 7763 64-Core Processors for a total of $128$ cores per socket, and $256$ GB of RAM.
The high-fidelity variance, estimated using $20$ MC samples, is $\sigma_0^2 \approx 0.0279$.

\subsubsection{A coarse-grid low-fidelity model}
We first consider a low-fidelity model $f^{(1)}$ obtained by coarsening the grid in $z$ direction (20 points), $v_{||}$ direction ($20$ points) and $\mu$ direction ($8$ Gauss-Laguerre points), for a total of $48,000 = 15 \times 20 \times 20 \times 8$ degrees of freedom, which corresponds to $5.76\times$ fewer degrees of freedom than the high-fidelity model.
We note that this represents still a useful resolution for this experiment while further considerable coarsening of the underlying mesh could lead to significant loss of accuracy or potentially even unphysical solutions.
The average evaluation cost of $f^{(1)}$ relative to the high-fidelity model is $w_1 = 0.0849$ on $16$ cores on one node on Lonestar6.
The variance and correlation coefficient, approximated using $20$ MC high- and low-fidelity model evaluations, are $\sigma_1^2 \approx 0.0256$ and $\rho_1 \approx 0.9990$, respectively.

Single-fidelity settings require fine discretizations for obtaining results with desired accuracy. 
Here, we want to show that when coarse-grid low-fidelity models are correlated with the high-fidelity model, their lower computational cost can be leveraged for speeding up multi-fidelity sampling methods. 
This is particularly relevant in plasma physics simulations in which usually only a limited number of high-fidelity simulations can be performed for training data-fit low-fidelity models.
Also, we want to show that our algorithm has a wide scope in the sense that a wide variety of low-fidelity models can be used.
 
\subsubsection{A sensitivity-driven dimension-adaptive sparse grid low-fidelity model}\label{sec:PlasmaSGModel}
The low-fidelity model $f^{(2)}_{n_2}$ is based on sensitivity-driven dimension-adaptive sparse grid interpolation \cite{Fa20, Farcas2020}, with code available\footnote{\url{https://github.com/ionutfarcas/sensitivity-driven-sparse-grid-approx}}.
Such sparse grid (SG) low-fidelity models have been successfully used for uncertainty quantification in plasma micro-instability simulations \cite{Fa20, Farcas2020, FDJ21}, including computationally expensive nonlinear simulations \cite{FMJ22}.
The adaptive procedure to construct the SG low-fidelity model is terminated when the underlying sparse grid reaches cardinality $n_2^*$, which corresponds to $n_2^*$ high-fidelity training samples, where $n_2^*$ is determined as discussed in Section~\ref{sec:camfmc}.
We note that even though the interpolation points used to construct the SG low-fidelity model have fine granularity (see \cite{Fa20, Farcas2020}), in cases where we cannot have exactly $n_2^*$ sparse-grid points, we take the closest, feasible number of sparse grid points to $n_2^*$. 

We use $20$ high- and low-fidelity model evaluations to asses the accuracy of the SG low-fidelity model; see also \cite{Ko22}.
To estimate the evaluation cost rate, we average the runtimes of the low-fidelity model evaluated at $1,000$ MC samples.
As can be seen in Figure~\ref{fig:AUG_SG_rates} and Table~\ref{tab:AUG_SG_rate_constants}, accuracy and runtime can be modeled well with an algebraic rate.
\begin{figure}[htbp]
\centering
\includegraphics[width=0.8\textwidth]{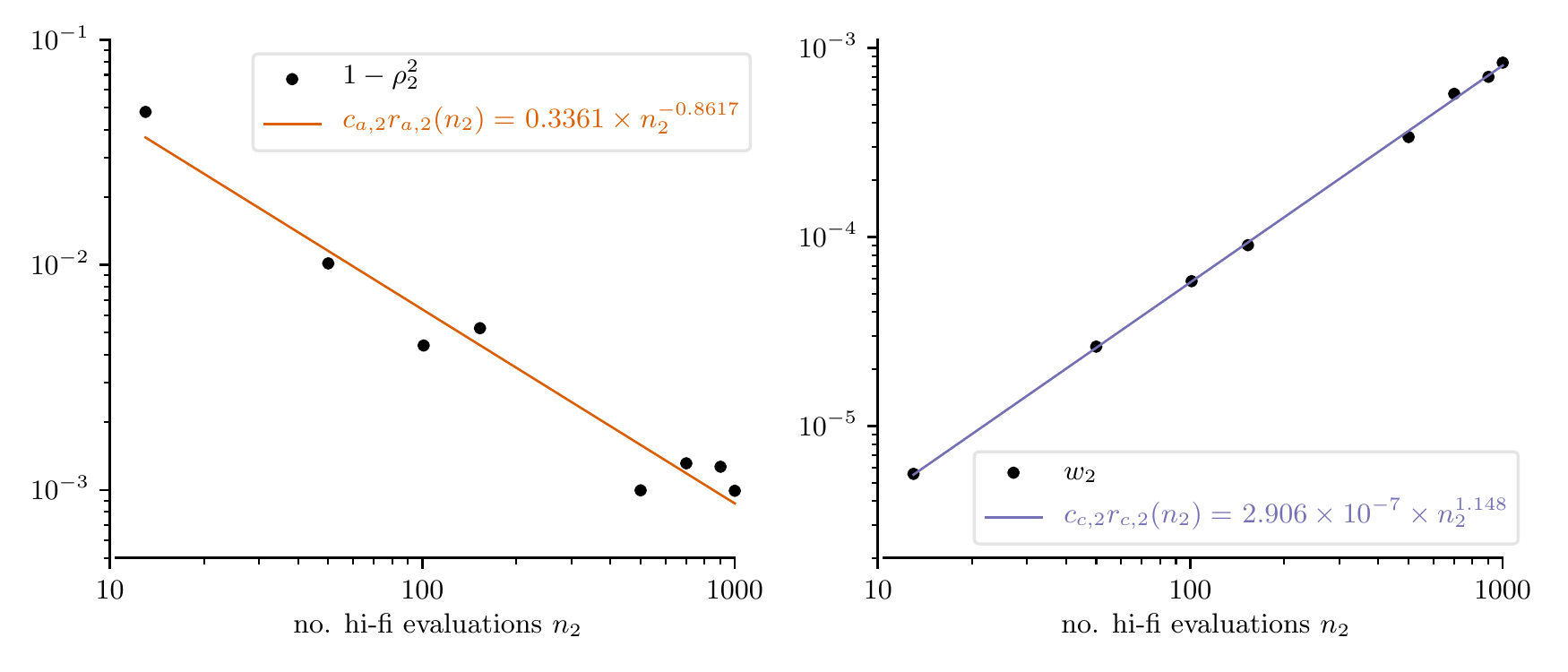}
\caption{ASDEX Upgrade Experiment: Estimated accuracy (left) and evaluation cost rates (right) for the SG low-fidelity model.}
\label{fig:AUG_SG_rates}
\end{figure}
\begin{table}[htbp]
\centering
\begin{tabular}{cc}
\begin{tabular}{|c|c|c|}
    \hline
    rate & rate type & estimated rate constants\bigstrut[t] \\
    \hline
    accuracy: $c_{a,2}r_{a,2}(n_2)$ & algebraic: $r_{a,2}(n_2) = n_2^{-\alpha_2}$ & $c_{a,2} = 0.3361$, $\alpha_2 = 0.8617$  \bigstrut[t]  \\
    \hline
    evaluation cost: $c_{c,2}r_{c,2}(n_2)$ & algebraic: $r_{c,2(n_2)} = n_2^{\beta_2}$ & $c_{c,2} = 2.9060 \times 10^{-7}$, $\beta_2 = 1.1480$ \bigstrut[t]\\
    \hline
\end{tabular} &
\end{tabular}
\caption{ASDEX Upgrade Experiment: Estimated accuracy and evaluation cost rates, and their constants for the SG low-fidelity model.}
\label{tab:AUG_SG_rate_constants}
\end{table}

\subsubsection{A deep learning low-fidelity model}

Low-fidelity model $f^{(3)}_{n_3}$ is a given by a fully connected, feed-forward three-layer deep neural network (DNN) with ReLU activation functions that is trained on high-fidelity data. 
The number of nodes of the hidden layers is $\lfloor 0.05 \times n_3 \rfloor$ and thus depends on the number of training samples $n_3$. 
Training is done over $20$ epochs using an Adam optimizer in which the loss function is the MSE.
The accuracy rate is estimated from $20$ high- and low-fidelity model evaluations and the cost rate from $1,000$ low-fidelity model evaluations; see Figure~\ref{fig:AUG_ML_rates} and Table~\ref{tab:AUG_ML_rate_constants}.
Notice that within the considered training size, the accuracy rate decreases rather slowly, at an algebraic rate of only about $0.2$.
This indicates that a large training size is required for obtaining accurate approximations should the DNN low-fidelity model be used to replace the high-fidelity model.
In addition, the evaluation cost rate increases very slowly, which is because the network has few layers only and the employed libraries are highly optimized. 
\begin{figure}[htbp]
\centering
\includegraphics[width=0.8\textwidth]{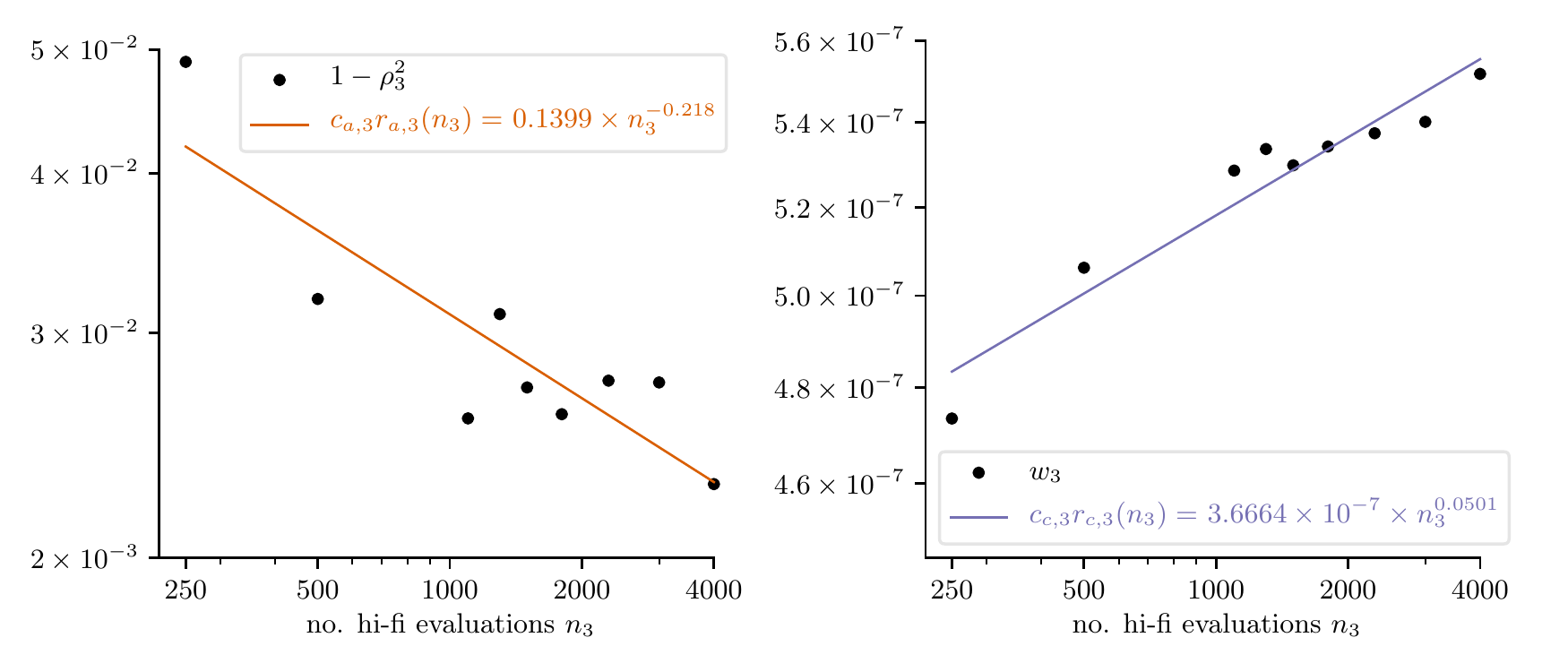}
\caption{ASDEX Upgrade Experiment: Estimated accuracy (left) and evaluation cost rates (right) for the DNN low-fidelity model.}
\label{fig:AUG_ML_rates}
\end{figure}
\begin{table}[htbp]
\centering
\begin{tabular}{cc}
\begin{tabular}{|c|c|c|}
    \hline
    rate & rate type & estimated rate constants\bigstrut[t] \\
    \hline
    accuracy: $c_{a,3}r_{a,3}(n_3)$ & algebraic: $r_{a,3}(n_3) = n_3^{-\alpha_3}$ & $c_{a,3} = 0.1399$, $\alpha_3 = 0.2180$  \bigstrut[t]  \\
    \hline
    evaluation cost: $c_{c,3}r_{c,3}(n_3)$ & algebraic: $r_{c,3}(n_3) = n_3^{\beta_3}$ & $c_{c,3} = 3.6664 \times 10^{-7}$, $\beta_3 = 0.0501$ \bigstrut[t]\\
    \hline
\end{tabular} &
\end{tabular}
\caption{ASDEX Upgrade Experiment: Estimated accuracy and evaluation cost rates, and their constants for the DNN low-fidelity model.}
\label{tab:AUG_ML_rate_constants}
\end{table}
    
\subsubsection{Context-aware multi-fidelity estimation}
We consider the computational budget $p = 5 \times 10^5$ seconds, which corresponds to $1,216$ high-fidelity model evaluations and is the limit of the computational resources available to us for this project.
To determine which combinations of models lead to CA-MFMC estimators with low MSEs for a budget $p$, we plot the analytic MSEs in Figure~\ref{fig:AUG_MSE_all}. 
The CA-MFMC estimator with the high-fidelity model $f^{(0)}$, the coarse-grid low-fidelity model $f^{(1)}$ and either the SG low-fidelity model $f^{(2)}_{n_2}$ or the DNN low-fidelity model $f^{(3)}_{n_3}$ lead to low MSEs. 
In contrast, the CA-MFMC estimator with the high-fidelity model $f^{(0)}$, SG model $f^{(2)}$, DNN model $f^{(3)}$ but without the coarse-grid model $f^{(1)}$ leads to a higher MSE. 
We therefore consider the CA-MFMC estimators with $f^{(0)}$, $f^{(1)}$ and either $f^{(2)}_{n_2}$ or $f^{(3)}_{n_3}$ (cf.~Section \ref{subsubsec:camfmc_multiple_models}).
This experiment shows that the analytic MSE can be used to select model combinations.
\begin{figure}[t]
\centering
\includegraphics[width=0.8\textwidth]{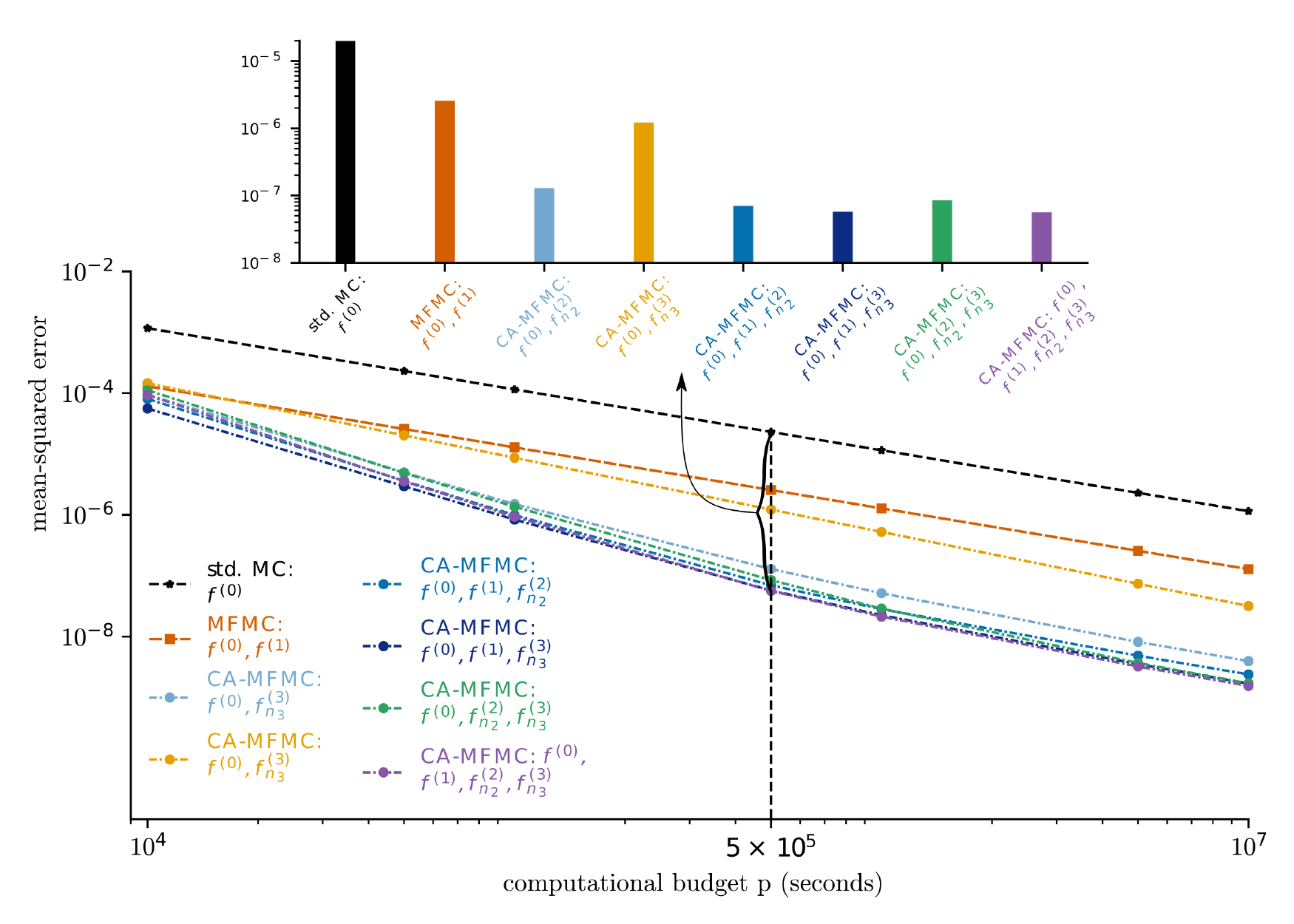}
\caption{ASDEX Upgrade Experiment: Combinations of models can be selected based on the analytical MSE, which can be computed without additional model evaluations as long as the accuracy and cost rates are available. In this example, the CA-MFMC estimator with the high-fidelity model $f^{(0)}$, the coarse-grid low-fidelity model $f^{(1)}$ and either the SG low-fidelity model $f^{(2)}$ or the DNN low-fidelity model $f^{(3)}$ have the highest (best) cost/error ratio, which indicates a good model selection.}
\label{fig:AUG_MSE_all}
\end{figure}

Based on the estimated constants in Table~\ref{tab:AUG_SG_rate_constants}, the accuracy and cost rates of the SG low-fidelity models are both strictly convex for any $n_2 \in [1, p  - 1]$ and any budget $p > 0$, and therefore Propositions~\ref{prop:MultiCAMFMC} and \ref{prop:BoundMulti} apply.
That is, the optimal, context-aware number of high-fidelity model evaluations for constructing the SG low-fidelity model, $n_2^*$, exists, is unique and bounded with respect to the budget. 
On the other hand, we see from Table~\ref{tab:AUG_ML_rate_constants} that the algebraic cost rate of the DNN low-fidelity model is concave since $\beta_3 < 1$, which means that we need to verify whether the assumptions in Propositions~\ref{prop:MultiCAMFMC} hold true.
To this end, we verify whether the function $h_3(n_3) = \kappa_{2} + \hat{c}_{a,3}r_{a,3}(n_3) + c_{c,3}r_{c,3}(n_3)$ defined in \eqref{eq:PK-1Kappa}, with $\kappa_{2} = w_1 = 0.0849$ and $\hat{c}_{a,3} = \kappa_1 c_{a, 3}$ with $\kappa_1 = 1 - \rho_1^2 = 0.0018$, satisfies condition \eqref{eq:ExtraStrictConvex}.
We do so for budget $p = 10^7$ seconds, i.e., the largest budget considered in Figure~\ref{fig:AUG_MSE_all}, for which $n_3 \in [1; 24,330]$.
The left plot in Figure~\ref{fig:AUG_u_h_2_ML} shows that $h_3^{\prime \prime}(n_3) > 0$ for all $n_2 \in [1, 24,330]$ which implies that \eqref{eq:ExtraStrictConvex} is satisfied for all budgets $0 < p \leq 10^7$, including the budget considered in our computations, i.e., $p = 5 \times 10^5$ seconds.
Therefore Propositions~\ref{prop:MultiCAMFMC} applies and the optimal, context-aware number of high-fidelity model evaluations for constructing the DNN low-fidelity model, $n_3^*$, exists and is unique. 
The right plot in Figure~\ref{fig:AUG_u_h_2_ML} depicts the objective $u_3$ defined in \eqref{eq:CAMFMC:ObjJ}, which has a unique global minimizer. 
\begin{figure}[htbp]
\centering
\includegraphics[width=1.0\textwidth]{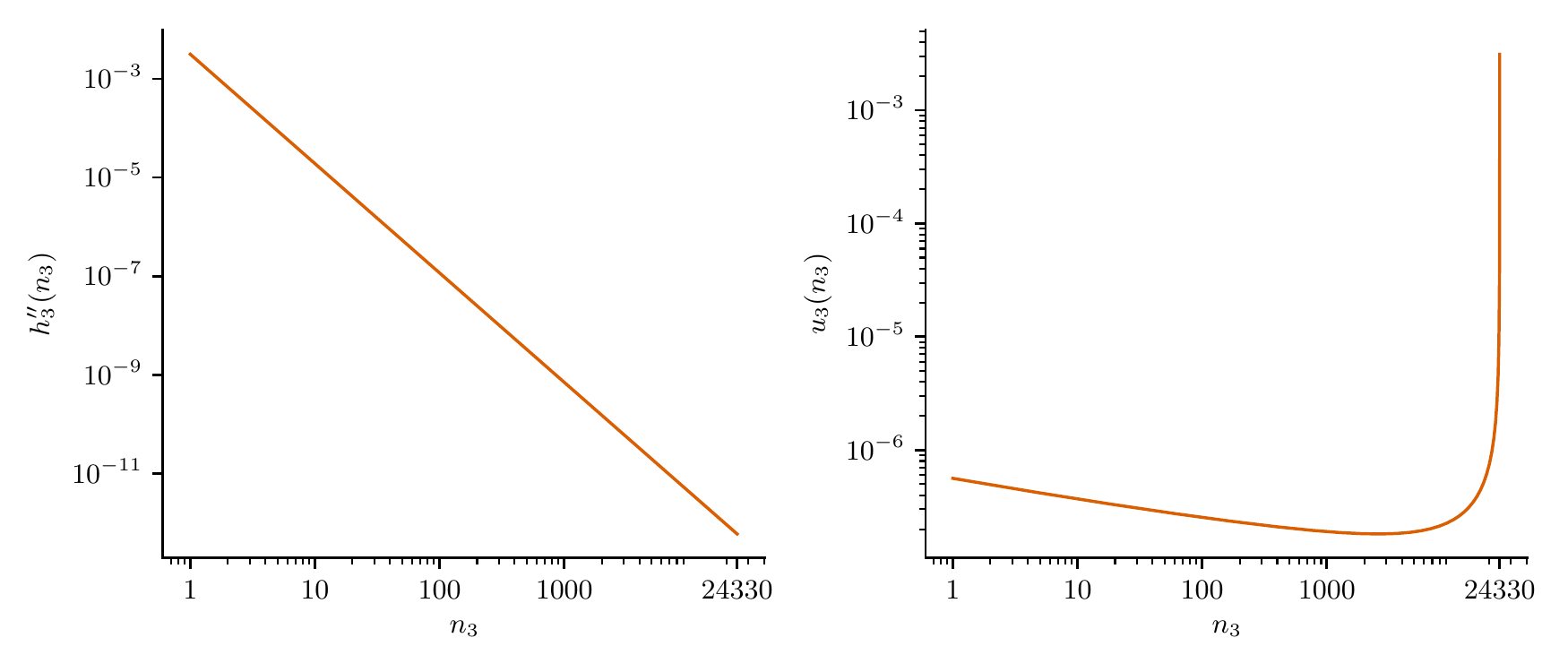}
\caption{ASDEX Upgrade Experiment:  On the left, we verify that $h''_3 > 0$ to
satisfy \eqref{eq:ExtraStrictConvex}) up to budget $p = 10^7$ seconds. 
On the right, we plot the objective $u_3$ defined in \eqref{eq:CAMFMC:ObjJ}, which has a unique global minimizer.} 
\label{fig:AUG_u_h_2_ML}
\end{figure}

In the left plot of Figure~\ref{fig:AUG_n_star}, we show the number of training samples $n_2^*$ for the CA-MFMC estimators with models $f^{(0)}, f_{n_2}^{(2)}$ and $f^{(0)}, f^{(1)}, f_{n_2}^{(2)}$. For the SG model $f_{n_2}^{(2)}$ we took $133$ training samples because the optimal $n_2^* = 131$ does not corresponding to a valid sparse grid; cf.~Section~\ref{sec:PlasmaSGModel}. 
In the right plot of Figure~\ref{fig:AUG_n_star}, the number of training samples $n_3^*$ for the CA-MFMC estimator with models $f^{(0)}, f^{(1)}, f_{n_3}^{(3)}$ is shown. Training with only $n_3^* = 159$ samples in case of our budget $p = 5 \times 10^{5}$ is sufficient to obtain a low-fidelity model that is useful for variance reduction in the CA-MFMC estimator together with the high-fidelity model, even though the low-fidelity model alone is insufficient to provide accurate predictions in single-fidelity settings. 
\begin{figure}[htbp]
\centering
\includegraphics[width=1.0\textwidth]{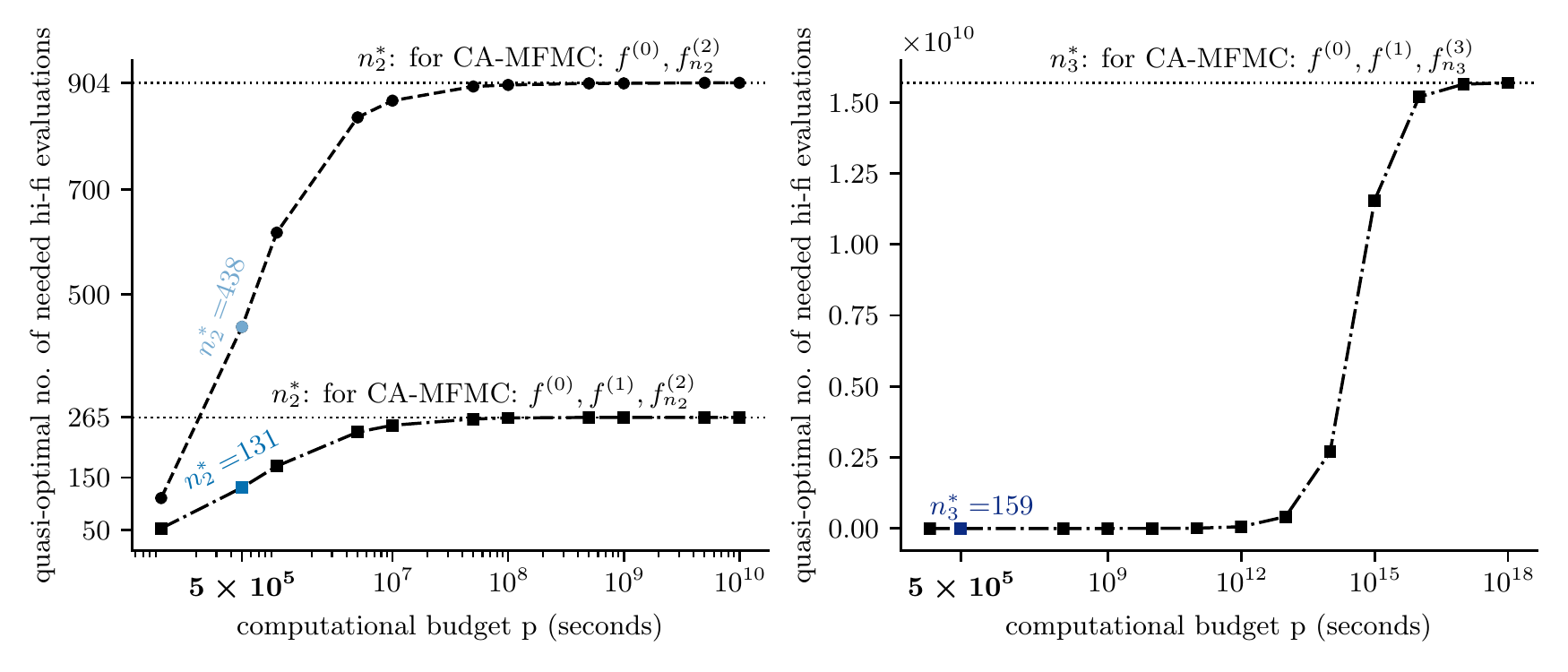}
\caption{ASDEX Upgrade Experiment: The number $n_2^*$ and $n_3^*$ of high-fidelity training samples select by the proposed context-aware approach for obtaining a multi-fidelity estimator with highest (best) cost/error ratio. For example, depending on whether the coarse-grid-based low-fidelity model $f^{(1)}$ is used together with the high-fidelity model $f^{(0)}$ and the SG-based model $f^{(2)}_{n_2}$, the proposed approach selects either $n_2^* = 131$ or $n_2^* = 438$ training samples, which demonstrates the context-aware trade-off that is made by the proposed approach when training low-fidelity models.}
\label{fig:AUG_n_star}
\end{figure}

Consider now the left plot of Figure~\ref{fig:AUG_MSEDist} that compares the MSEs of the CA-MFMC estimators to static MFMC and regular MC. 
The reference solution was obtained via the MFMC estimator depending on the coarse-grid and DNN low-fidelity models for a budget of $p = 3 \times 10^6$ seconds using $N = 5$ replicates; here we needed $n_3^* = 847$ high-fidelity evaluations to construct the DNN low-fidelity model. 
The average reference value of the expected growth rate of the dominant electron temperature gradient-driven micro-instability mode is $\hat{\mu}_{\mathrm{ref}} = 0.3967$.
The context-aware estimators are up to two orders of magnitude more accurate than the standard MC estimator. The roughly two orders of magnitude speedup of the CA-MFMC estimators translate into a runtime reduction from $72$ days to roughly $4$ hours on $32$ cores on one node of the Lonestar6 supercomputer. We also observe that MFMC with the high-fidelity and coarse-grid model is about one order of magnitude more accurate in terms of the MSE than the standard MC estimator, showing that coarse-grid plasma physics models can be useful in multi-fidelity settings. The trend that we observe here numerically are in alignment with the analytic MSEs shows in Figure~\ref{fig:AUG_MSE_all}. 

The right plot of Figure~\ref{fig:AUG_MSEDist} shows how the number of samples are distributed among the models. All multi-fidelity estimators allocate more than $98 \%$ of the total number of samples to low-fidelity models, which explains the speedups.
However, there is still a small percentage corresponding to the high-fidelity model, which is desirable in practical applications to achieve unbiasedness of the estimators. 
\begin{figure}[htbp]
\centering
\includegraphics[width=1.0\textwidth]{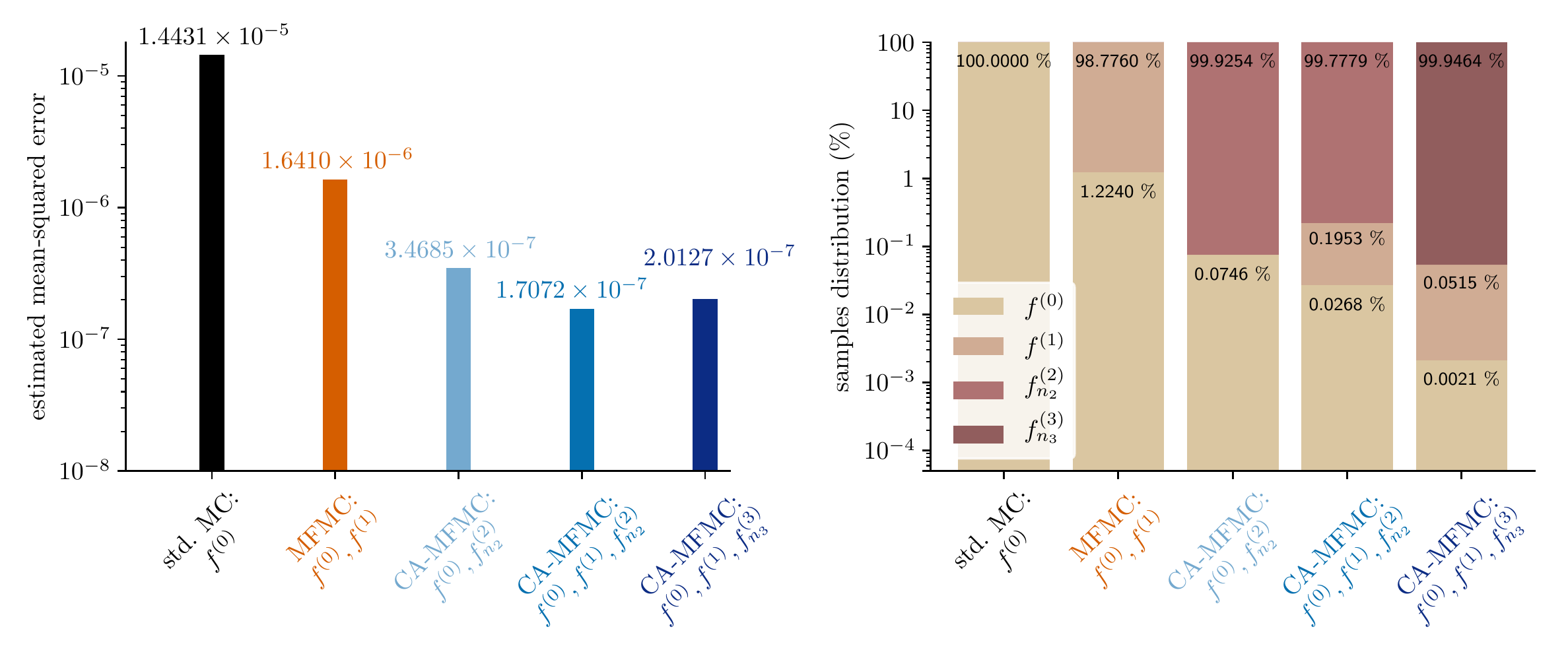}
\caption{ASDEX Upgrade Experiment: The left plots show that the CA-MFMC estimators achieve about two orders of magnitude speedup, which translates into a runtime reduction from $72$ days to roughly $4$ hours on one node of the Lonestar6 supercomputer. The right plot shows the sample distribution across models in the employed estimators.}
\label{fig:AUG_MSEDist}
\end{figure}

\section{Conclusions} \label{sec:conclusion}
The proposed context-aware learning approach for variance reduction trades off training hierarchies of low-fidelity models with Monte Carlo estimation.
By leveraging the trade-off between either adding a high-fidelity model sample to the training set for constructing low-fidelity models versus using the high-fidelity model sample in the Monte Carlo estimator, our analysis shows that the quasi-optimal number of training samples from the high-fidelity model is bounded independent of the computational budget and thus that after a finite number of samples, all high-fidelity model samples should be used in the Monte Carlo estimator rather than being used to improve the low-fidelity model. 
This means that low-fidelity models can be over-trained for multi-fidelity uncertainty quantification. 
We have demonstrated the over-training numerically in a thermal block scenario, in which a reduced basis low-fidelity model constructed from $50$ high-fidelity training samples led to a multi-fidelity estimator with a poorer cost/error ratio than the proposed context-aware estimator that determined $18$ to be the maximum quasi-optimal number of training samples. 
This has clear implications on training data-fit and machine-learning-based models too, which can require large data sets for achieving high prediction accuracy but are sufficient to be trained with few data points when they are merely used for variance reduction together with the high-fidelity model in multi-fidelity uncertainty quantification, as we have shown in a numerical experiment with plasma micro-turbulence simulations.

\section*{Acknowledgements}
I.-G.F.~and F.J.~were partially supported by the Exascale Computing Project (No.~17-SC-20-SC), a collaborative effort of the U.S.~Department of Energy Office of Science and the National Nuclear Security Administration.
B.P.~acknowledges support from the Air Force Office of Scientific Research (AFOSR) award FA9550-21-1-0222 (Dr.~Fariba Fahroo).
We thank Tobias Goerler for useful discussions and insights about the considered plasma micro-turbulence simulation scenario.
We also gratefully acknowledge the compute and data resources provided by the Texas Advanced Computing Center at The University of Texas at Austin (\texttt{https://www.tacc.utexas.edu/}).

\bibliography{ca_mfmc}
\bibliographystyle{abbrv}

\end{document}